\documentclass[11pt,final]{article} 
\usepackage{etex}
\usepackage[utf8]{inputenc} 

\def\bibsection{\section*{References}}

\usepackage{geometry} 
\usepackage{graphicx} 

\usepackage{booktabs} 
\usepackage{array} 
\usepackage{paralist} 
\usepackage{verbatim} 
\usepackage{subfig} 
\usepackage{natbib}
\usepackage{hyperref}

\usepackage{fancyhdr} 
\usepackage{sectsty}
\allsectionsfont{\sffamily\mdseries\upshape} 

\usepackage[nottoc,notlof,notlot]{tocbibind} 
\usepackage[titles,subfigure]{tocloft} 

\usepackage{amsmath}
\usepackage{amssymb}
\usepackage{amsthm}
\usepackage{mathrsfs}
\usepackage{enumerate}
\usepackage[all]{xy}
\usepackage{todonotes}
\usetikzlibrary{calc}

\newtheorem{thm}{Theorem}[subsection]
\newtheorem{prop}[thm]{Proposition}
\newtheorem{cor}{Corollary}[thm]

\newtheorem{lem}[thm]{Lemma}
\theoremstyle{definition}
\newtheorem{prob}[thm]{Problem}

\newtheorem{defn}[thm]{Definition}
\newtheorem{ex}[thm]{Example}

\newtheorem{remk}[thm]{Remark}

\newcommand{\ol}[1]{\overline{#1}}
\newcommand{\wt}[1]{\widetilde{#1}}

\newcommand{\vocab}[1]{\textbf{#1}}

\newcommand{\bbr}{\mathbb R}
\newcommand{\bbz}{\mathbb Z}

\newcommand{\bbn}{\mathbb N}

\newcommand{\bbc}{\mathbb C}

\newcommand{\diff}{\backslash}

\newcommand{\tr}{\textnormal{tr}}

\newcommand{\sheaf}[1]{\mathcal{#1}}

\newcommand{\weyl}{\Omega}
\newcommand{\weyll}{{\widehat{\Omega}}}

\newcommand{\mweyl}{{M_N(\Omega)}}
\newcommand{\mweyll}{{M_N(\widehat{\Omega})}}

\newcommand{\mxx}[4]{\left(\begin{array}{cc} #1 & #2\\ #3 & #4\end{array}\right)}

\title{Elementary Examples of Solutions to Bochner's Problem for Matrix Differential Operators}
\author{W.R. Casper}

\begin{document}

\maketitle
\begin{abstract}
In this paper, we demonstrate an elementary method for constructing new solutions to Bochner's problem for matrix differential operators from known solutions.  We then describe a large family of solutions to Bochner's problem, obtained from classical solutions, which include several examples known from the literature.  By virtue of the method of construction, we show how one may explicitly identify a generating function for the associated sequence of monic $w$-orthogonal matrix polynomials $\{p(x,n)\}$, as well as the associated algebra $D(w)$ of all matrix differential operators for which the $\{p(x,n)\}$ are eigenfunctions.  We also include some general results on the structure of the algebra $D(w)$.
\end{abstract}

\section{Introduction}
A weight function is a nonzero measureable function $w: \bbr\rightarrow [0,\infty)$ satisfying the condition that the moments $\int w(x)x^n dx$ are all finite.  A weight function gives rise to an inner product on the space of polynomials in $x$, defined by
$$\langle f(x),g(x)\rangle_w := \int f(x)w(x)g(x)^*dx.$$
By Gram-Schmidt, we may obtain a sequence of pairwise orthogonal polynomials $\{p(x,n)\}$ such that $p(x,n)$ has degree $n$ for each integer $n\geq 0$.  Such a sequence is called a sequence of orthogonal polynomials for $w$, or if $w$ is implied, simply a sequence of orthogonal polynomials.

Bochner's problem, introduced in the paper \cite{bochner1929sturm}, is to determine for which weights $w$ the associated polynomials are a family of eigenfunctions of some second order differential operator.  Lucky for us, Bochner provides a solution to his problem in the very same paper.  Up to affine changes of coordinates, the only orthogonal polynomials satisfying this property are the classical families of Hermite, Laguerre, and Jacobi.

The question Bochner addresses generalizes naturally to operators of higher order.  Consider a weight function $w$ and a sequence of orthogonal polynomials $\{p(x,n)\}$ . The set $D(w)$ of all differential operators for which the sequence is a family of eigenfuntions forms an algebra.  Moreover, this algebra is independent of the choice of sequence of orthogonal polynomials for $w$, since after all the value of $p(x,n)$ is unique up to a scalar multiple for each $n$.  In terms of $D(w)$, Bochner's result tells us for which $w$ the algebra $D(w)$ contains an operator of order two.  However, the classification of weight matrices $w$ for which the algebra $D(w)$ contains an operator of higher order seems to be very hard.  The classification for operators of order $4$ was done by Krall \cite{krall1981orthogonal}, but for higher order is still open.

Bochner's problem and the notion of orthogonal polynomials also extends naturally to matrix-valued polynomials.  A weight matrix is a matrix-valued function $w(x): \bbr\rightarrow M_N(\bbc)$ which is sufficiently nice (explained below) so as to induce a nondegenerate matrix-valued inner product on the space of matrix-valued polynomials defined by
$$\langle f(x),g(x)\rangle_w := \int f(x)w(x)g(x)^*dx.$$
Here $g(x)^*$ refers to the Hermitian-conjugate of $g(x)$.  Again by a process identical to Gram-Schmidt, but with matrix values, we may obtain a sequence of polynomials $\{p(x,n)\}$ such that $p(x,n)$ is degree $n$ with \emph{nonsingular} leading coefficient for all integers $n\geq0$ and such that $\langle p(x,n),p(x,m)\rangle_w = 0I$ for $m\neq n$.  A sequence of matrix-valued polynomials satisfying these properties is called a sequence of orthogonal polynomials for $w$.

To extend Bochner's problem to matrix orthogonal polynomials, one should consider the algebra $M_N(\Omega)$ of matrix-valued differential operators, eg. operators of the form
$$\delta =  a_0(x) + \partial a_1(x) + \dots + \partial^\ell a_\ell(x),$$
where the $a_i(x)\in M_N(\bbc[[x]])$ and $\delta$ acts on $M_N(\bbc[[x]])$ by
$$f(x)\cdot\delta := f(x)a_0(x) + f'(x)a_1(x) + \dots + f^{(\ell)}(x)a_\ell(x).$$
In particular, our differential operators act on the right.  This is required because of the noncommutativity of the coefficients: the differential operators must act on the right in order to be compatible with the matrix-valued inner product \cite{duran1997matrix}.  A function $f(x)$ is called an eigenfunction of $\delta$ if there exists a matrix $\lambda\in M_N(\bbc)$ such that $f(x)\cdot\delta = \lambda f(x)$.  With this notion of matrix differential operators, we may state the matrix version of Bochner's problem.
\begin{prob}[Bochner's Problem for Matrix Differential Operators]
Let $w(x)$ be a weight matrix, and let $\{p(x,n)\}$ be a sequence of orthogonal matrix polynomials for $w(x)$.  When does there exists a matrix differential operator of order two for which $p(x,n)$ is an eigenfunction for every integer $n\geq 0$?
\end{prob}

Bochner's problem for matrix differential operators is considered in numerous papers, including many of the papers mentioned in the references below.  For a helpful survey, see \cite{duran2005survey}.  Unlike the scalar case, general classification results for Bochner's problem remain elusive, even for $2\times 2$ matrices.  Many papers therefore have focused instead on providing new examples of Bochner pairs $(w,\delta)$, eg. a weight matrix $w$ and a second order operator $\delta\in D(w)$.

More recent papers have explored the structure of the algebra $D(w)$ of all matrix differential operators for which the $p(x,n)$ are eigenfunctions.  In particular, \cite{tirao2011} \cite{castro2006} \cite{grunbaum2007} provide examples of generators and relations for $D(w)$ for various values of $w$.  Even more examples are \cite{duran2008some}\cite{grunbaum2008}\cite{pacharoni2008sequence}\cite{zurrian2015}\cite{zurrian2016algebra}.  These examples demonstrate that the structure of $D(w)$ can be nuanced and interesting, unlike in the scalar case.  However, despite an ever increasing list of examples, more general results regarding the structure of $D(w)$ remain a mystery.  In particular, current methods of finding generators and relations for $D(w)$ are often ad hoc, or based on computational evidence, and can involve extended calculation.
 
The purpose of the current paper is to demonstrate how under sufficiently nice conditions one may use Darboux transformations to create a new Bochner pair $(\wt w,\wt \delta)$ from a known Bochner pair $(w,\delta)$.  Moreover, we show that when $(\wt w,\wt \delta)$ arises from $(w,\delta)$ by a Darboux transformation then the algebras $D(\wt w)$ and $D(w)$ are closely related, so that knowledge of the structure of $D(w)$ leads to knowledge of the structure of $D(\wt w)$.  As a result, this paper leads both to new families of examples of Bochner pairs not currently in the literature, as well as very efficient derivations of the structure of the algebra $D(\wt w)$ of matrix differential operators for several examples of weights $\wt w$ already featured in the literature.  Lastly, we prove a couple general results regarding the structure of the algebra $D(w)$, not currently featured in the literature.  In particular, we prove that $D(w)$ must necessarily be finitely generated over its center, and that its center must be an affine curve.  We also prove that for $N = 2$, the center must be rational.

A Darboux transformation of a differential operator $\delta$ is a new differential operator $\wt\delta = \mu\nu$ obtained by means of a factorization of the original operator $\delta = \nu\mu$.  For a short survey of Darboux transformations, see \cite{rosu1998short}, and for a great read about the role of Darboux transformations in the context of orthogonal polynomials satisfying differential equations, see \cite{griinbaum1996orthogonal}.  Further background on the Darboux transformations relevant to this paper can be found in \cite{grunbaum2011darboux}\cite{etingof1997factorization}.  Finally, for a very recent article exploring Darboux transformations in a noncommutative context, applying in particular to matrix differential operators, see \cite{geiger2015noncommutative}.  This notion works equally as well for matrix differential operators as it does for ordinary differential operators.  However, given an arbitrary factorization of $\delta$, there's no reason to expect that the transformed operator $\wt\delta$ will be in $D(\wt w)$ for some new weight matrix $\wt w$.  In order to guarantee this, we must be more methodical about our choice of the factorization of $\delta$.

The method presented in this paper relies on the application of two different adjoints $*$ and $\dag$ on the algebra of differential operators.  The first adjoint $*$ is the standard notion of the adjoint, which extends the Hermitian-conjugate on matrix-valued functions and satisfies
$$(\partial^m)^* = (-1)^m\partial^m.$$
The second $\dag$ is the formal $w$-adjoint of $\delta$, and is defined in terms of the first by 
$$\delta^\dag := w(x)\delta^*w^{-1}(x),$$
in some neighborhood of $x=0$.  The requisite properties of $w$ are discussed further in the section on adjoints later in the paper.  For clarity, we provide a formula for the formal $w$-adjoint of a first order differential operator in the next example.
\begin{ex}
Consider the first-order matrix differential operator $\delta = a_0(x) + \partial a_1(x)$.  Then the standard adjoint of $\delta$ is given by
$$\delta^* = a_0(x)^* - a_1'(x)^* - \partial a_1(x)^*,$$
and the formal $w$-adjoint of $\delta$ is given by
$$\delta^\dag = w(x)[a_0(x)^* - a_1'(x)^* - w(x)^{-1}w'(x)a_1(x)^*]w(x)^{-1} - \partial w(x)a_1(x)^*w(x)^{-1}.$$
\end{ex}

Using the two adjoint definitions we have introduced, we are now able to state our Main Theorem, which provides a way of obtaining a Darboux transformation of a Bochner pair $(w,\delta)$ to a new Bochner pair $(\wt w,\wt \delta)$.
\begin{thm}[Main Theorem]\label{main theorem}
Let $(w,\delta)$ be a Bochner pair supported on $(x_0,x_1)$ with $\lim_{x\rightarrow x_i} (|x|^n + 1)w(x) = 0$ for $i=0,1$ and $n\geq 0$ and
$$\delta = a_0(x) + \partial a_1(x) + \partial^2 a_2(x)I,$$
with $a_0(x),a_1(x)$ matrix-valued functions satisfying $a_1(x)^\dag = a_1(x)$, with $a_2(x)$ a scalar-valued function which is nonzero on $(x_0,x_1)$, and with $\delta$ degree-preserving.  Suppose that there exist matrix polynomials $v_1(x),v_0(x)$ satisfying the following properties:
\begin{itemize}
\item  $\deg(v_i(x)) = i$, $\det(v_i(x))\neq 0$ on $(x_0,x_1)$ for $i=0,1$ and $(v_0v_1^{-1})^\dag = v_0v_1^{-1}$
\item  $\tr((v_i(x)a_2(x)v_i(x)^\dag)^{-1})\in L^2(\tr(w(x))dx)$ for $i=0,1$
\item  $(v_0(x)v_1(x)^{-1})^{2}a_2(x) + (v_0(x)v_1(x)^{-1})'a_2(x) + (v_0(x)v_1(x)^{-1})a_1(x) + a_0(x) = 0$
\end{itemize}
Then the following is true
\begin{enumerate}[(a)]
\item  there exists a smooth, matrix-valued function $f$ on $(x_0,x_1)$ satisfying $v_1f(v_1f)^\dag = a_2$
\item  $\wt w = fwf^*$ is a weight matrix
\item  $\nu := \partial v_1(x) - v_0(x)$ and $\mu := -f(x)f^\dag(x)\nu^\dag$ are degree-preserving and $\delta = \nu\mu$
\item  $\wt p(x,n) := p(x,n)\cdot\nu$ is a sequence of orthogonal matrix polynomials for $\wt w$
\item  $(\wt w,\wt \delta)$ is a Bochner pair for $\wt\delta := \mu\nu$
\item  $$D(\wt w)\supseteq (\nu^{-1}D(w)\nu)\cap M_N(\Omega) = \{\nu^{-1}\eta\nu: \eta\in D(w),\ \ker(\nu)\cdot\eta\subseteq\ker(\nu)\}.$$
\end{enumerate}
\end{thm}
The Main Theorem is not as general as possible.  The idea is to decompose $\delta$ as $\delta = -\nu ff^\dag\nu^\dag$ such that $\nu$ is degree-preserving and $w$-adjointable.  The Main Theorem simply provides a specific context where this holds.  Even more generally, we could find $\nu$ such that $\delta\nu = \nu\wt\delta$ for some differential operator $\wt\delta$, eg. a Darboux conjugation rather than a Darboux transformation.
It is also important to note that the function $f$ in the Main Theorem is not unique -- and in fact there are many such functions up to a choice of smooth, unitary matrix-valued function.  To find such an $f$, we split $w(x)$ as $u(x)u(x)^*$, which we may do since $w(x)$ is Hermitian.  Then $u^{-1}v^{-1}a_2 (v^\dag)^{-1} u$ is also Hermitian, and therefore may be split as $hh^*$ for some matrix-valued function $h$.  Then taking $f = uhu^{-1}$ completes the construction. 

Motivated by our Main Theorem, we make the following definition
\begin{defn}\label{Darboux transformation definition}
Let $(w,\delta)$ be a Bochner pair, and let $\{p(x,n)\}$ be a sequence of orthogonal matrix polynomials for $w$.  We call $(\wt w,\wt\delta)$ a \vocab{Darboux transformation} of $(w,\delta)$ if there exist matrix differential operators $\mu$ and $\nu$ such that $\delta = \nu\mu$, $\wt\delta = \mu\nu$, and $\wt p(x,n) := p(x,n)\cdot\nu$ defines a sequence of orthogonal matrix polynomials for $\wt w$.  We also say that $\nu$ is a Darboux transformation from the Bochner pair $(w,\delta)$ to the Bochner pair $(\wt w,\wt\delta)$.  We define the \vocab{associated subalgebra} $D(\wt w,\nu,w)$ of $D(\wt w)$ by
$$D(\wt w,\nu,w) := (\nu^{-1}D(w)\nu)\cap\mweyl = \{\nu^{-1}\delta\nu: \delta\in D(w),\ \ker(\nu)\cdot\delta\subseteq \ker(\nu)\}.$$
\end{defn}
In general, $D(\wt w,\nu,w)$ will be a proper subalgebra of $D(\wt w)$.  However, we should expect the subalgebra to contain \emph{most} of $D(\wt w)$.  In particular, we have the following proposition, the proof of which is found in § \ref{The Proof Section}.
\begin{prop}\label{subalgebra proposition}
Let $\nu$ be a Darboux transformation from the Bochner pair $(w,\delta)$ to the Bochner pair $(\wt w,\wt\delta)$.  Suppose $\wt\delta = \partial^2\wt a_2 + \partial \wt a_1 + \wt a_0$ with $\det(\wt a_2)\neq 0$ and that there exists a positive integer $\ell$ such that
$$\dim_\bbc( D(\wt w,\nu,w)_i/D(\wt w,\nu, w)_{i-1}) = \dim_\bbc( D(\wt w,\nu, w)_{i+2}/D(\wt w,\nu,w)_{i+1}),\ \forall i\geq\ell.$$
where here $D(\wt w,\nu,w)_i$ represents the linear subspace of $D(\wt w,\nu,w)$ of operators of order at most $i$.  Then $D(\wt w)$ is a finitely generated algebra over $D(\wt w, \nu, w)$, and is generated by elements of order less than $\ell$.
\end{prop}
In practice, the previous proposition will allow us to determine the structure of $D(\wt w)$ from the structure of $D(w)$.

\subsection{A Brief Reading Guide}
Our Main Theorem simultaneously doubles as a source of new examples of weight matrices $w$ with nontrivial $D(w)$, and a means of obtaining explicitly the value of the algebra $D(w)$.  As a demonstration of this, we include explicit examples of calculations of $D(w)$ from the literature, as well as extensions to higher-dimensional analogues of these examples.  This paper itself attempts to present some very analytic ideas in an algebraic way, with the motivation of attracting the algebraically minded to the problem of studying the structure of $D(w)$ from the perspective of noncommutative algebra and algebraic geometry.  For this reason, the author includes a more in-depth background in sections § \ref{classical orthogonal polynomials} and § \ref{matrix orthogonal polynomials}.  The reader who is already acquainted with the basic ideas and notation of our topic, and who would appreciate a dessert-before-dinner approach to the results, is encouraged to skip forward to the examples in § \ref{example section}.   The reader who is less acquainted with the theory of classical orthogonal polynomials and the current theory of their matrix counterparts is encouraged to peruse the background sections § \ref{classical orthogonal polynomials} and § \ref{matrix orthogonal polynomials}.  Of course, we encourage all readers to glance at § \ref{terminology and notation} whenever they are uncertain about terminology or notation.

\subsection{Terminology and Notation}\label{terminology and notation}
As a basic notation used throughout the paper:
\begin{itemize}
\item  $I$ the identity matrix (with size determined by context)
\item  $\bbc[[x]]$ the algebra of power series in $x$
\item  $\weyl$ the Weyl algebra, ie. the algebra of differential operators with $\bbc[x]$-coefficients
\item  $\weyll$ the algebra of differential operators with $\bbc[[x]]$-coefficients
\item  $M_N(R)$ the algebra of $N\times N$-matrices with entries in $R$
\item  $D(w)$ the algebra of differential operators for weight matrix $w$ (see \ref{D(w) definition})
\item  $E(w)$ the algebra of eigenvalue sequences for weight matrix $w$ (see \ref{E(w) definition})
\item  $\ker(\delta)$ the kernel of a matrix differential operator $\delta\in\mweyl$ as a linear operator on the $\bbc$-vector space $M_N(\bbc[[x]])$
\item  $\nu^{-1}$ the inverse of $\nu\in\mweyl$ as a matrix pseudo-differential operator (if it exists)
\item  $1_{(x_0,x_1)}(x)$ the indicator function of the interval $(x_0,x_1)$
\end{itemize}

We recall the definition of the Weyl algebra $\weyl$ for this paper.
\begin{defn}\label{Weyl algebra definition}
The Weyl algebra $\weyl$ is defined to be the set
$$\weyl = \{a_0(x) + \partial a_1(x) + \dots + \partial^n a_n(x): n\in\bbn,\ a_1(x),\dots, a_n(x)\in \bbc[x]\}$$
where the products are defined by means of the fundamental commutation relation
$$x\partial - \partial x = 1.$$
\end{defn}
Note that this is slightly at odds with the typical definition of the Weyl algebra, whose fundamental commutation relation is $\partial x - x\partial  = 1$.  This is because, for reasons explained later, our matrix differential operators will act on functions on the \emph{right}.  As an explicit example, consider a second-order differential operator
$$\delta = a_0(x) + \partial a_1(x) + \partial^2 a_2(x)\in\weyl.$$
It will act on a smooth function $f(x)$ by
$$f(x)\cdot\delta = f(x)a_0(x) + f'(x)a_1(x) + f''(x)a_2(x).$$
This right action has the consequence of reversing the usual fundamental commutation relation of the Weyl algebra.

\section{Background}
\subsection{Classical Orthogonal Polynomials and Bochner's Problem}\label{classical orthogonal polynomials}
This section is intended to provide a reader with a quick recap of the basic points in the theory of orthogonal polynomials.  We begin with the definition of a weight function.  For simplicity, we dodge technical measure-theoretic details by working with smooth weights.

\begin{defn}
Let $-\infty\leq a<b\leq \infty$.  A \vocab{weight function} supported on $(x_0,x_1)$ is a nonnegative function $r:\bbr\rightarrow (0,\infty)$ which is $0$ off of $(x_0,x_1)$ and positive and smooth on $(x_0,x_1)$, satisfying the condition that the moments $\int_\bbr x^n r(x)dx <\infty$ for all integers $n\geq 0$.  The interval $(x_0,x_1)$ on which $r$ is nonzero is called the \vocab{support of $r$}.
\end{defn}

Given a weight function $r(x)$, we define an inner product on $\bbc[x]$ by
$$\langle p(x),q(x)\rangle_r := \int_\bbr p(x)r(x)\ol{q(x)}dx.$$
By the process of Gram-Schmidt elimination, we can construct a sequence of pairwise-orthogonal polynomials $p(x,0),p(x,1),p(x,2),\dots$ with $p(x,n)$ a polynomial of degree $n$.  If we impose the additional constraint that each $p(x,i)$ is monic, then this sequence of polynomials is unique.

\begin{defn}
Let $r(x)$ be a weight function.  A sequence of polynomials $\{p(x,n)\}_{n=0}^\infty$ is called a \vocab{sequence of orthogonal polynomials} for $r$ if $p(x,n)$ has degree $n$ for all integers $n\geq 0$ and $\langle p(x,n),p(x,m)\rangle_r = 0$ for all $m,n\geq 0$ with $n\neq m$.
\end{defn}

The sequence of monic orthogonal polynomials for a weight function $r(x)$ can in practice be calculated recursively.
\begin{prop}[Three-Term Recurrence Relation]
Let $r(x)$ be a weight function.   Then for all integers $n\geq 0$, there exist constants $s_n,t_n\in\bbc$ such that
\begin{equation}\label{recurrence relation}
xp(x,n) = p(x,n+1) + s_np(x,n) + t_np(x,n-1).
\end{equation}
\end{prop}
This The values of $s_n$ and $t_n$ may be defined explicitly in terms of the moments of $r(x)$.

Sequences of orthogonal polynomials that are simultaneously eigenfunctions of a second-order differential operator arise naturally in Sturm-Liouville theory.  Three families of weights whose corresponding orthogonal polynomias satisfy a second-order differential operator were known classically.  These weights are listed in the table in Figure \ref{cop table I}.  The sequences of orthogonal polynomials for these weights are referred to as the classical orthogonal polynomials.
\begin{figure}[h]
\begin{center}
\begin{tabular}{|l|c|}\hline
Family    & weight                     \\\hline
Hermite   & $e^{-x^2}$                 \\
Laguerre  & $x^be^{-x}1_{(0,\infty)}$, $b>-1$\\
Jacobi    & $(1-x)^a(1+x)^b1_{(-1,1)}$, $a,b>-1$ \\\hline
\end{tabular}
\caption{The Classical Orthogonal Polynomial Weights}
\label{cop table I}
\end{center}
\end{figure}

The importance and applicability of the classical orthogonal polynomials lead Bochner to ask and answer the following question:
\begin{prob}[Bochner's Problem]\label{Bochner's Problem}
For which weight functions $r(x)$ do there exists second-order differential operators $\delta\in \weyl$ for which the sequence of monic orthogonal matrix polynomials $p(x,n)$ are simultaneously eigenfunctions?
\end{prob}
With this question in mind, we make the following definition.
\begin{defn}\label{scalar Bochner pair definition}
The pair of data $(r,\epsilon)$ with $r$ a weight function and $\epsilon$ a differential operator is called a \vocab{Bochner pair} if the sequence of monic orthogonal polynomials for $r$ are all eigenfunctions for $\epsilon$.
\end{defn}
In other words, $(r,\epsilon)$ is a Bochner pair if and only if there exists a sequence of complex numbers $\{\lambda_n\}_{n=0}^\infty$ such that $p(x,n)\cdot \epsilon = \lambda_np(x,n)$ for all $n\geq 0$.  The next theorem characterizes solutions to Bochner's problem.  See \cite{totik2005orthogonal}.
\begin{thm}\label{history lesson}
Suppose that $(r,\epsilon)$ is a Bochner pair, with $r$ supported on $(x_0,x_1)$ for $-\infty\leq a<b\leq \infty$.  Without loss of generality, we may assume that $\epsilon = \partial^2q(x) + \partial(l(x) + q'(x))$ for some polynomials $l(x), q(x)$ of degree $1$ and $2$, respectively.  Then the following are true:
\begin{enumerate}[(i)]
\item  the weight function $r(x)$ satisfies the Pearson equation
\begin{equation}\label{pearson equation}
q(x)r'(x)r^{-1}(x) = l(x),\ \ x\in (x_0,x_1)
\end{equation}
\item  the polynomials $q(x)$ and $l(x)$ satisfy the boundary condition
$$\lim_{x\rightarrow a+} q(x)r(x) = \lim_{x\rightarrow b-} l(x)r(x) = 0.$$
\end{enumerate}
Conversely, any weight function $r(x)$ supported on an interval $(x_0,x_1)$ satisfying (i) and (ii) for some polynomials $l(x)$ and $q(x)$ of degree $1$ and $2$, respectively gives a solution $(r,\epsilon)$ to Bochner's problem.
\end{thm}

In the case that $(r,\epsilon)$ is a Bochner pair, we can obtain the sequence of monic orthogonal matrix polynomias for $r$ in a far sleeker way than the recurrence relation of Equation \eqref{recurrence relation} \cite{totik2005orthogonal}.
\begin{thm}
Let $r(x)$ be a weight function satisfying the assumptions of Theorem \ref{history lesson}, and let $q(x)$ and $l(x)$ be as in the statement of the theorem.  Then the sequence of monic orthogonal polynomials $p(x,n)$ may be obtained by the Rodrigues formula
\begin{equation}\label{rodrigues formula}
p(x,n) = c_n (q(x)^n r(x))^{(n)} r^{-1}(x)
\end{equation}
for some constants $c_n$.  Moreover, the generating function
$$F(x, t) = \sum_{n=0}^\infty \frac{p(x,n)}{n!c_n}t^n$$
is analytic in $(x_0,x_1)\times I$ for some open interval $I$ of $0$ and is given by
\begin{equation}
F(x,t) = \frac{r(\lambda)}{r(x)}(1-tq'(\lambda))^{-1},
\end{equation}
where $\lambda = \lambda(x,t)$ satisfies
$$\lambda-x-tq(\lambda) = 0.$$
\end{thm}

The answer to Bochner's question is that the classical orthogonal polynomials are the only solutions to Bochner's problem.  This is the content of the next theorem.
\begin{thm}[Bochner\cite{bochner1929sturm}]
Up to an affine change of coordinates, the only weight functions $r(x)$ which are solutions of Bochner's problem are the classical weight functions.
\end{thm}

We can extend Bochner's problem by asking what other differential operators $\delta$ have the monic orthogonal polynomials of $r$ as common eigenfunctions.  Let $\{p(x,n)\}$ be the sequence of monic orthogonal polynomials for some weight function $r$.  The set
$$D(r) = \{\delta\in\weyl: p(x,n)\ \text{is an eigenfunction of $\delta$ for all $n$}\}$$
forms a subalgebra of the algebra of differential operators.  For $(r,\epsilon)$ a Bochner pair, it turns out that the answer is very straightforward.
\begin{thm}[Miranian\cite{miranian2005matrix}\cite{miranian2005classical}]\label{scalar algebra result}
If $(r,\epsilon)$ is a Bochner pair, then
$$D(r) = C_\weyl(\epsilon) = \bbc[\epsilon].$$
\end{thm}
\begin{remk}
We point out that $C_\weyll(\epsilon)$ may properly contain $C_\weyl(\epsilon)$.  For example, the operator $\epsilon = \partial^2(1-x^2) -\partial x$ commutes with the first-order operator $\partial (1-x^2)^{1/2}$, but the coefficients of the latter operator are not polynomial.
\end{remk}

The previous theorem shows that a second-order differential operator $\epsilon$ forming part of the Bochner pair $(r,\epsilon)$ generates the algebra $D(r)$, and is unique up to a scalar multiple.  Values of $\epsilon$ for various weight functions are listed in the table in Figure \ref{cop table II}.
\begin{figure}[h]
\begin{center}
\begin{tabular}{|l|c|}\hline
Weight   & $2$nd-order operator                         \\\hline
Hermite  & $\partial^2-2\partial x$                     \\
Laguerre & $\partial^2x + \partial(b + 1-x)$            \\
Jacobi   & $\partial^2(1-x^2) + \partial(b-a-(a+b+2)x)$ \\\hline
\end{tabular}
\caption{Classical Solutions to Bochner's Problem}
\label{cop table II}
\end{center}
\end{figure}

\subsection{Matrix Orthogonal Polynomials and Bochner's Problem}\label{matrix orthogonal polynomials}
We next review the basic theory of orthogonal matrix polynomials and Bochner's problem.
\begin{defn}\label{weight matrix definition}
A \vocab{weight matrix $w$} supported on an interval $(x_0,x_1)$ is defined to be a function $w: \bbr\rightarrow M_N(\bbr)$ satisfying the condition that $w$ vanishes outside of $(x_0,x_1)$, that everywhere in $(x_0,x_1)$ the matrix $w$ is entrywise-smooth, Hermitian, and positive-definite, and finally the condition that $w$ has finite moments:
$$\int x^m w(x)dx < \infty,\ \ \forall m\geq 0.$$
The interval $(x_0,x_1)$ where the matrix is nonsingular is called the \vocab{support of $w$}.
\end{defn}
Note that we work only with ``smooth" weight matrices in order to avoid more delicate analytic considerations.  For a survey on the analytic theory of matrix orthogonal polynomials, see \cite{damanik2008analytic}.  Every weight matrix $w$ defines a matrix-valued inner product $\langle \cdot,\cdot\rangle_w$ on the set $M_N(\bbc[x])$ of all $N\times N$ complex matrix polynomials by
$$\langle p,q\rangle_w := \int_\bbr p(x)w(x)q(x)^* dx,\ \ \forall p,q\in M_N(\bbc[x]).$$
Though it will not play an important role in our paper, the matrix-valued inner product above begets a traditional (scalar-valued) inner product by taking trace $\tr(\langle p,q\rangle_w)$.

A Gram-Schmidt type argument with our matrix-valued inner product shows that for each integer $n\geq 0$ there exists a polynomial $p_n$ of degree $n$ with nonsingular leading coefficient, uniquely defined up to its leading coefficient, such that $\langle p_n,p_m\rangle_w = 0I$ for all $m\neq n$.   We observe that orthogonality with respect to the matrix-valued inner product is a stronger condition than orthogonality with respect to $\tr(\langle p,q,\rangle_w)$.
\begin{defn}
We call a sequence of matrix polynomials $\{p(x,n)\}$ a \vocab{sequence of orthogonal matrix polynomials (OMP) for $w$} if $\deg(p(x,n)) = n$ for all integers $n\geq 0$, with non-singular leading coefficient, and
$$\langle p_n,p_m\rangle_w = 0I,\ \ \forall m\neq n.$$
\end{defn}
A OMP $\{p(x,n)\}$ for $w$ will be called \vocab{monic} if it satisfies the additional constraint that $p(x,n)$ has leading coefficient equal to the identity matrix $I$ for each fixed $n$.  Once again, there exists a unique sequence of monic OMP $\{p(x,n)\}$ for any fixed $w$.

Generalizing Favard's theorem, Dur\'an and L\'opez-Rodriguez proved that a sequence of polynomials $\{p(x,n)\}$ being the sequence of monic orthogonal matrix polynomials for some weight matrix $w$ is equivalent to the sequence $\{p(x,n)\}$ satisfying a three-term recurrence relation under certain regularity conditions.
\begin{thm}[Dur\'an and L\'opez-Rodriguez \cite{laredolectures}]
Suppose that $\{p(x,n)\}$ is a sequence of monic orthogonal matrix polynomials for a weight matrix $w$.  Then there exist sequences of complex-valued matrices $\{s_n\}$ and $\{t_n\}$ such that
\begin{equation}\label{matrix recurrence relation}
xp(x,n) = p(x,n+1) + s_np(x,n) + t_np(x,n-1),\ \forall n\geq 1
\end{equation}
Conversely, given reasonably nice sequences of marices $\{s_n\},\{t_n\}$, there exists a weight matrix $w$ for which the sequence of polynomials $\{p(x,n)\}$ defined by Equation \eqref{matrix recurrence relation} is a sequence of monic orthogonal matrix polynomials.
\end{thm}
Intuitively, the sequences $\{s_n\}$ and $\{t_n\}$ determine the moments of the weight matrix, and vice versa.  In this way, there is a sort of dictionary between weight matrices and sequences of polynomials satisfying a recurrence relation of the above form \cite{sinap1996}\cite{grunbaum2008}\cite{duran2005structural}\cite{duran2007structural}\cite{duran2005survey}.

Matrix polynomials $p(x)$ have a natural action by matrix differential operators, ie. differential operators whose coefficients are matrix polynomials.  This action is on the right in order to make it compatible with the matrix-valued inner product defined by $w$ and the corresponding three-term recursion relation with coefficients on the left.  In particular, for polynomials $a_0(x),a_1(x),a_2(x)\in M_N(\bbc[x])$ the second-order matrix differential operator $\epsilon = a_0(x) + \partial a_1(x) + \partial^2 a_2(x)$ acts on a matrix polynomial $p(x)\in M_N(\bbc[x])$ by
$$p(x)\cdot\epsilon = p(x)a_0(x) + p'(x)a_1(x) + p''(x)a_2(x).$$
With this in mind, Bochner's problem was reposed by Dur\'{a}n \cite{duran1997matrix}\cite{duran2004orthogonal} as the following
\begin{prob}[Bochner's Problem for Matrix Differential Operators]\label{Bochner's Matrix Problem}
Determine all weight matrices $w(x)$ such that there exists a second-order matrix differential operator $\delta\in \mweyl$ for which the associated sequence of monic orthogonal matrix polynomials are simultaneously eigenfunctions.
\end{prob}
For reasons that will be discussed below, when such an $\delta$ exists it may be taken to be $w$-symmetric.  This motivates the following definition.
\begin{defn}\label{matrix Bochner pair definition}
We define a (matrix) \vocab{Bochner pair} to be a pair $(w,\delta)$ with $w$ a weight matrix and $\delta\in\mweyl$ a $w$-symmetric second-order differential operator for which the sequence of monic OMP for $w$ are simultaneously eigenfunctions.
\end{defn}
Thus we are interested in the question of what Bochner pairs exist.  More generally, we are interested in the question of calculating the algebra $D(w)$ of all matrix differential operators having the monic OMP of $w$ as simultaneous eigenfunctions.
\begin{defn}\label{D(w) definition}
Let $w(x)$ be a weight matrix, and let $p(x,n)$ be the associated sequence of monic orthogonal matrix polynomials.  Then the \vocab{algebra of matrix differential operators associated to $w$}, denoted $D(w)$, is the set of all $\delta\in\mweyl$ such that $p(x,n)$ is an eigenfunction of $\delta$ for all $n$.
\end{defn}

The behavior in the scalar case ($N=1$) leads us to ask questions such as whether the solutions to Bochner's problem for matrix differential operators must satisfy a form of Pearson's equation or a Rodrigues-type recurrence relation.  Dur\'{a}n and Gr\"{u}nbaum \cite{duran2004orthogonal} provide a partial result in this direction in the following theorem, restated in terms of formal adjoints.
\begin{thm}[Dur\'{a}n-Gr\"{u}nbaum]
Let $w$ be a weight matrix.  Then $D(w)$ contains a second-order differential operator if and only if there exists
$$\delta = \partial^2a_2 + \partial a_1 + a_0,$$
with $a_i\in M_N(\bbc[x])$ and $\deg(a_i)\leq i$ for all $i$ satisfying
\begin{enumerate}[(a)]
\item  $a_2w = wa_2^*$
\item  $w\cdot\delta^* = a_0w$ for $\delta^*$ the formal adjoint of $\delta$
\item  $a_2w$ and $(a_2w)' - a_1w$ vanish as $x$ approaches the endpoints of the support of $w$
\end{enumerate}
In this case, $(w,\delta)$ is a solution to Bochner's problem for matrix differential operators.
\end{thm}
Note that the first two conditions translate to the statement that $\delta$ is formally $w$-symmetric, while the third condition implies that the formal $w$-adjoint of $\delta$ is an adjoint (see Def. \ref{formal adjoint}).  In other words, the content of the above theorem is that a weight matrix $w$ has a second-order differential operator in $D(w)$ if and only if there exists a $w$-symmetric second-order differential operator.  The above conditions imply that $w$ satisfies the non-commutative Pearson equation
\begin{equation}\label{noncommutative pearson equation}
2(a_2w)' = a_1w + wa_1^*.
\end{equation}

\subsection{Adjoints of Differential Operators}
The study of adjoints of differential operators can be considered a subset of the study of adjoints of unbounded linear operators on certain well-chosen Hilbert spaces.  For this reason, we briefly recall the definition of an unbounded linear operator and its adjoint.  For a great reference on adjoints of unbounded linear operators, as well as adjoints of differential operators, see \cite{dunford1963linear}.
\begin{defn}
Let $\sheaf H$ be a Hilbert space.  An \vocab{unbounded linear operator} on $\sheaf H$ is a linear function $T: \sheaf D(T)\rightarrow V$, where $\sheaf D(T)$ is a dense subspace of $\sheaf H$, called the \vocab{domain of $T$}.  The adjoint of $T$ is an unbounded linear operator $T^*$ defined on the set
$$\sheaf D(T^*) := \{y\in \sheaf H: x\rightarrow \langle Tx,y\rangle\ \text{is continuous}\}$$
and satisfying $\langle Tx,y\rangle = \langle x,T^*y\rangle$ for all $x\in D(T),y\in D(T^*)$.
\end{defn}
Note that the fact that $D(T)$ is dense is important for the existence of $T^*$, which then exists as a consequence of the Hahn-Banach theorem.

Differential operators make great examples of unbounded linear operators.  An interesting question is when for given differential operator the adjoint is also a differential operator.
\begin{ex}
For example, consider the first-order differential operator $\delta = \partial a_1(x) + a_0(x)$ with $a_0(x),a_1(x)\in\bbc[x]$.  Given a weight function $r(x)$ supported on $(x_0,x_1)$, we may view this as an unbounded linear operator on the inner product space $L^2(r(x)dx)$, induced by the action of $\delta$ on $\bbc[x]$.

Using integration by parts, one finds that for polynomials $p(x),q(x)$:
\begin{align*}
\langle p(x)\cdot\delta,q(x)\rangle_r
  & = p(x_1)a_1(x_1)r(x_1)q(x_1)-p(x_0)a_0(x_0)r(x_0)q(x_0)\\
  & + \langle p(x),q(x)\cdot (-a_1(x)^*r(x)\partial a_1(x)^{-1} + a_0^*(x))\rangle_r.
\end{align*}
Thus if $a_1(x_0)r(x_0) = a_1(x_1)r(x_1) = 0$, we see that the adjoint of $\delta$ is the differential operator $\delta^\dag =-a_1(x)^*r(x)\partial a_1(x)^{-1} + a_0^*(x) $.  In particular, if $r(x)$ is constant on its support, and $a_1(x_i) = 0$ for $i=0,1$, then the adjoint $\delta^\dag$ of $\delta$ is $\delta^\dag = -a_1(x)^*\partial + a_0^*(x)$.
\end{ex}
A similar observation applies to matrix orthogonal polynomials, and this motivates the definition of a $*$-operation on the algebra of matrix differential operators.

The algebra $\mweyll$ of matrix differential operators with coefficients in the power series ring $\bbc[[x]]$ is equipped with a canonical $*$-operation, extending the usual $*$-operation on $M_N(\bbc[[x]])$ defined by Hermitian conjugates and satisfying
$$(xI)^* = xI,\ \ \ (\partial I)^* = -\partial I.$$

\begin{defn}\label{formal adjoint}
Let $w$ be a weight matrix supported on an interval $(x_0,x_1)$ containing $0$ and let $\delta\in \mweyll$.  A \vocab{formal adjoint of $\delta$ with respect to $w$}, or \vocab{formal $w$-adjoint of $\delta$} is the unique differential operator $\delta^\dag\in \mweyll$ defined on $(x_0,x_1)$ by
$$\delta^\dag := w\delta^* w^{-1}.$$
\end{defn}
(Note that this definition relies on the assumption that $w$ is Hermitian.)  In particular, the $*$-operation is exactly the formal $w$-adjoint for $w(x) = 1_{(x_0,x_1)}(x)I$.  The assumption that $0$ is in $(x_0,x_1)$ may always be achieved by means of an affine change of coordinates, and is necessary for $w$ to have a local representation as a power series based at $0$.  Moreover, the assumption that $w$ is positive-definite also implies that $w^{-1}$ has a power series expansion at $0$.  Consequently, the formal $w$-adjoint is indeed an element of $\mweyll$.  As a notational point, we will use $*$ to denote the canonical adjoint on $\mweyll$ and will use $\dag$ to denote the formal adjoint with respect to a particular weight matrix $w$ (the value of $w$ will be implied from the context).
\begin{defn}
Let $w$ be a weight matrix, and $\delta\in\mweyll$ a differential operator.  Then $\delta$ is called \vocab{formally $w$-symmetric} if $\delta^\dag = \delta$.
\end{defn}

We next define the adjoint of a differential operator.  To do so, we must define the Hilbert space upon which it acts.
\begin{defn}
Let $w$ be a weight matrix.  We define the \vocab{Hilbert space of $w$} to be
$$\sheaf H(w) := M_N(L^2(\tr(w)dx)).$$
We call a matrix differential operator $\delta\in \mweyll$ \vocab{amenable} if $p(x)\cdot\delta\in \sheaf H(w)$ for all matrix-valued polynomials $p(x)\in M_N(\bbc[x])$.
\end{defn}

\begin{defn}
Let $w$ be a weight matrix and let $\delta\in\mweyll$ be an amenable matrix differential operator.  An amenable matrix differential operator $\wt\delta\in\mweyl$ satisfying the identity
$$\langle p\cdot\delta,q\rangle_w = \langle p,q\cdot\wt \delta\rangle_w,\ \forall p,q\in M_N(\bbc[x]),$$
is called an \vocab{adjoint of $\delta$ with respect to $w$}, or \vocab{$w$-adjoint of $\delta$}.  If $\delta = \wt\delta$, then $\delta$ is called \vocab{$w$-symmetric}.
\end{defn}

Every differential operator has a formal $w$-adjoint, but not necessarily a $w$-adjoint.  Put another way, even though each differential operator $\delta$ will have an adjoint as an unbounded operator on $\sheaf H(w)$, this linear operator adjoint need not be a differential operator.
\begin{prop}
Let $w$ be a weight matrix and let $\delta\in\mweyll$ be an amenable matrix differential operator.  Then an adjoint of $\delta$ with respect to $w$, if one exists, is equal to the formal $w$-adjoint $\delta^\dag$.
\end{prop}

\begin{ex}
Consider the weight function $r(x) = e^{-x}1_{(0,\infty)}(x)$ supported on the interval $(0,\infty)$.  The formal $r$-adjoint of the differential operator $\partial$ is
$$\partial^\dag = e^{-x}\partial^*e^x = e^{-x}(-\partial)e^x = -\partial + 1.$$
However,
$$\langle 1\cdot\partial, 1\rangle_r = \langle 0,1\rangle_r = 0$$
and
$$\langle 1,1\cdot \partial^\dagger\rangle_r = \langle 1,1\rangle_r = 1.$$
Therefore $\partial^\dagger$ is not an $r$-adjoint for $\partial$ and thus $\partial$ has no $r$-adjoint.

As another example, the formal adjoint of the operator $\partial x$ is
$$(\partial x)^\dag = e^{-x}(\partial x)^*e^x = e^{-x}(-x\partial)e^x = -\partial x + x-1.$$
Moreover, integration by parts shows us that
$$\langle p'(x)x,q(x)\rangle_r = \langle p(x),-q'(x)x + q(x)(x-1)\rangle_r$$
for all polynomials $p(x),q(x)\in\bbc[x]$.  Therefore $-\partial x + x-1$ is in fact an $r$-adjoint of $\partial x$.
\end{ex}

Expanding on these examples, we have the following lemma
\begin{lem}\label{first order adjoint}
Let $\nu = \partial f_1(x) + f_0(x)$ be a matrix differential operator for $f_1,f_0\in M_N(\bbc[x])$.  If $f_1(x_i)w(x_i) = 0$ for $i=0,1$ then $\nu$ has a $w$-adjoint.
\end{lem}
\begin{proof}
For any polynomials $p(x),q(x)\in M_N(\bbc[x])$, integration by parts tells us that
$$\langle p(x)\cdot\nu,q(x)\rangle_w = p(x)f_1(x)w(x)q^*(x)|_{x_0}^{x_1} + \langle p(x),q(x)\cdot\nu^\dag\rangle.$$
Thus if $f_1(x_i)w(x_i) = 0$ for $i=0,1$ the formal $w$-adjoint $\nu^\dag$ of $\nu$ is the $w$-adjoint.
\end{proof}

The application of adjoints of differential operators to our situation is provided by the following theorem.
\begin{thm}[Gr\"{u}nbaum-Tirao \cite{grunbaum2007}]
Let $w$ be a weight matrix, and let $\delta\in D(w)$.  Then a $w$-adjoint of $\delta$ exists and is in $D(w)$.  The operator $\delta\mapsto\delta^\dag$ is an involution on $D(w)$ giving $D(w)$ the structure of a $*$-algebra.
\end{thm}
As a consequence, we have the following corollary originally proved in \cite{grunbaum2007}.
\begin{cor}
Let $w$ be a weight matrix.  Then $D(w)$ contains a differential operator of order $m$ if and only if $D(w)$ contains a $w$-symmetric differential operator of order $m$.
\end{cor}
\begin{proof}
If $D(w)$ contains a second-order matrix differential operator $\omega$, then by the previous theorem $D(w)$ contains $\omega + \omega^\dag$ and $i(\omega - \omega^\dag)$.  Both operators are clearly symmetric, and by definition of the adjoint of order at most $m$.  Since
$$\omega = \frac{1}{2}(\omega + \omega^\dag) - \frac{1}{2}i(i(\omega - \omega^\dag))$$
and $\omega$ is of order $m$, at least one of the two symmetric operators must be order $m$.
\end{proof}

\section{Proof of the Main Theorem}
\subsection{Degree-Preserving Differential Operators}
\begin{defn}
We call a matrix differential operator $\delta\in \mweyl$ \vocab{degree-filtration preserving} if for all polynomials $q(x)\in M_N(\bbc[x])$, the degree of $q(x)\cdot\delta$ is no greater than the degree of $q(x)$.  We call $\delta$ \vocab{degree-preserving} if the degree of $q(x)\cdot\delta$ is equal to the degree of $q(x)$ for all $q(x)\in M_N(\bbc[x])$.  In particular, degree-preserving differential operators necessarily act injectively on the algebra of matrix polynomials.
\end{defn}
We use $\mweyl^F$ to denote the subalgebra of $\mweyl$ of degree-filtration preserving matrix differential operators.  Due to its distinguished role in the following, we fix the notation $s = \partial x$.
\begin{lem}
The subalgebra $\mweyl^F$ is equal to the $M_N(\bbc)$-subalgebra of $\mweyl$ generated by $s$ and $\partial$
\end{lem}
\begin{proof}
Note that both $s$ and $\partial$ are degree-filtration preserving, and therefore the $M_N(\bbc)$-subalgebra of $\mweyl$ that they generate is contained in $\mweyl^F$.  Thus to prove our lemma, it suffices to show the opposite containment.  Suppose that $\delta$ is degree-filtration preserving, of order $n$.  Since $\delta\in \mweyl$, we know that
$$\delta = \sum_{i=0}^n \partial^ia_i(x)$$
for some $a_i(x)\in M_N(\bbc[x])$.

We claim that $\deg(a_i(x))\leq i$ for all $0\leq i\leq n$.  To see this, suppose otherwise.  Then let $j$ be the smallest nonnegative integer satisfying $\deg a_j(x) > j$.  Then
$$x^j I \cdot\delta = \sum_{i=0}^j \frac{j!}{(j-i)!} x^{j-i}a_i(x)$$
is a polynomial of degree greater than $j$, contradicting the assumption that $\delta$ is degree-filtration preserving.  This proves our claim.

Next note that for all integers $j\geq 1$,
$$\partial^jx^j = s(s-1)(s-2)\dots(s-j+1).$$
Therefore if $a(x)\in M_N(x)$ is of degree $\leq j$, then $\partial^ja_j(x)$ is in the $M_N(\bbc)$-subalgebra of $\mweyl$ generated by $\partial$ and $s$.  From this it follows that $\delta$ is in the subalgebra of $\mweyl$ generated by $s$ and $\partial$.  This proves our lemma.
\end{proof}
\begin{remk}
The expression $\partial^jx^j = s(s-1)(s-2)\dots(s-j-1)$ found in the proof above is exactly the reason why the substitution $x = e^t$ may be used to re-express an Euler-Cauchy equation as a linear ordinary differential equation with constant coefficients.
\end{remk}

\begin{prop}
The natural map
$$\bigoplus_{n=0}^\infty \partial^n M_N(\bbc[s]) \rightarrow \mweyl^F$$
is an isomorphism of $M_N(\bbc[s])$-modules.  In particular, every $\delta\in \mweyl^F$ has a natural expression of the form
$$\delta = \sum_{i=0}^n\partial^ia_i(s)$$
for some integer $n\geq 0$ and some matrices $a_i(s)\in M_N(\bbc[s])$.
\end{prop}
\begin{proof}
Note that $s\partial^j = \partial^j (s+j)$, and therefore $a(s)\partial^j = \partial^ja(s+j)$ for all $a(s)\in M_N(\bbc[s])$.  Using this relation, the above follows.
\end{proof}

\begin{prop}
Let $w$ be a weight matrix, and let $p(x,n)$ be a sequence of monic orthogonal matrix polynomials for $w$.  If $\delta\in D(w)$, then $\delta$ is degree-filtration preserving.  Consequently $\delta =\sum_{i=0}^n\partial^i a_i(s)$, for some $a_i(s)\in M_N(\bbc[s])$, in which case
$$p(x,n)\cdot \delta = a_0(n)p(x,n),\ \forall n\in\bbz_+.$$
\end{prop}
\begin{proof}
Suppose that $\delta\in D(w)$.  Then for all integers $i\geq 0$, there exists $\lambda_i\in M_N(\bbc)$ such that $p(x,i)\cdot\delta = \lambda_ip(x,i)$.  Furthermore, since the $p(x,i)$ form a basis for $M_N(\bbc[x])$, given a $q(x)\in M_N(\bbc[x])$ of degree $n$, we may write $q(x) = \sum_{i=0}^n c_i p(x,i)$ for some $c_0,\dots c_n\in M_N(\bbc)$.  Then
$$q(x)\cdot\delta = \sum_{i=0}^n c_i \lambda_i p(x,i),$$
which has degree at most $n$.  Since $q(x)$ was arbitrary, this shows that $\delta$ is degree-filtration preserving.  Therefore by the previous proposition, $\delta = \sum_{i=0}^r\partial^i a_i(s)$ for some $a_i(s)\in M_N(\bbc[s])$.

We know that $p(x,n) = Ix^n + \text{(lower degree terms)}$, and therefore
$$p(x,n)\cdot\delta = \sum_{i=0}^r p(x,n)\cdot \partial^ia_i(s) = a_0(n)x^n + \text{(lower degree terms)}.$$
Therefore since $p(x,n)\cdot\delta = \lambda_np(x,n) = \lambda_nx^n + \text{(lower degree terms)}$, we have that $\lambda_n = a_0(n)$ for all integers $n\geq 0$.
\end{proof}

The previous proposition in particular shows that if $p(x,n)$ is a sequence of monic orthogonal matrix polynomials for a weight matrix $w$ and if $\delta\in D(w)$, then the sequence $\{\lambda_n\}_{n=0}^\infty\subseteq M_N(\bbc)$ satisfying $\lambda_np(x,n) = p(x,n)\cdot\delta$ will be a polynomial in $n$.  Thus we may define a map $\Lambda: D(w)\rightarrow M_N(\bbc[n])$ satisfying the property
$$\Lambda(\delta)(n)p(x,n) = p(x,n)\cdot\delta.$$
We denote the image of the map $\Lambda$ in $M_N(\bbc[n])$ by $E(w)$.
\begin{defn}\label{E(w) definition}
We call the subalgebra $E(w)$ the \vocab{algebra of eigenvalue sequences associated to the weight matrix $w$}.  We call the map $\Lambda$ the \vocab{eigenvalue isomorphism}.
\end{defn}
The eigenvalue homomorphism defines an njection of $D(w)$ into $M_N(\bbc[n])$.  This was shown in \cite{grunbaum2007}, but is reproved here in our notation.
\begin{prop}
The map $\Lambda$ is injective, and in particular defines an isomorphism of $D(w)$ onto $E(w)$.
\end{prop}
\begin{proof}
Suppose that $\Lambda(\delta) = \Lambda(\delta') = \lambda(n)\in M_N(\bbc[n])$ for some $\delta,\delta'\in D(w)$.  Then this in particular implies that for all $n$,
$$p(x,n)\cdot(\delta-\delta') = (\lambda(n)-\lambda(n))p(x,n) = 0.$$
It follows that the kernel of $\delta-\delta'$ contains $M_N(\bbc)\{p_0,p_1,\dots\} = M_N(\bbc[x])$.  Therefore $\delta-\delta'$ must be identically zero.
\end{proof}

\subsection{Kernels of Differential Operators}
In this section we touch on some necessary results regarding the kernel of a matrix differential operator of order $r$.  By default throughout the paper, unless stated otherwise, the kernel of a differential operator always refers to its kernel as a linear operator on the ring of power series $M_N(\bbc[[x]])$.  
\begin{defn}
Let $\mu\in \mweyll$ be a matrix differential operator.  We define the \vocab{kernel of $\mu$} to be the subspace $\ker(\mu)\subseteq M_N(\bbc[[x]])$ defined by
$$\ker(\mu) = \{\psi\in M_N(\bbc[[x]]): \psi\cdot\mu = 0I\}.$$
\end{defn}
For this reason, the elements of the kernel may not in fact belong to the Hilbert space determined by the weight matrix $w$.  Furthermore, in this section the term matrix differential operator will always refer to an element in $\mweyll$.
\begin{lem}\label{first kernel lemma}
Suppose that $\mu$ is a monic matrix differential operator of order $r$.  Then $\ker(\mu)$ is a locally free as a left $M_N(\bbc)$-module, generated by some $\psi_1,\dots,\psi_r\in M_N(\bbc[[x]])$ with $\psi_i$ a unit in $M_N(\bbc[[x]])$ for all $i$.
\end{lem}
\begin{proof}
There exist $a_0(x),\dots, a_{r-1}(x)\in M_N(\bbc[[x]])$ such that $\mu = \partial^rI + \sum_{i=0}^{r-1} \partial^i a_i(x)$.  Then the matrix differential equation
$$y^{(r)}(x) + \sum_{i=0}^{r-1}y^{(i)}(x)a_i(x) = 0$$
may be linearized to 
$$Y' = YU,\ 
Y = \left(\begin{array}{cccc}
y & y' & \dots & y^{(r-1)}
\end{array}\right),\ 
U = \left(\begin{array}{ccccc}
0 & 0 & \dots & 0 & -a_0(x)\\
I & 0 & \dots & 0 & -a_1(x)\\
0 & I & \dots & 0 & -a_2(x)\\
\vdots & \vdots & \ddots & \vdots & \vdots\\
0 & 0 & \dots & I & -a_{r-1}(x)
\end{array}\right)$$
for $Y\in \bbc[[x]]^{N\times rN}$, $U\in \bbc[[x]]^{rN\times rN})$:
Let $U_0,U_1,\dots\in M_{rN}(\bbc)$ with $U = \sum_{i=0}^\infty U_ix^i$.  Then $Y=\sum_{i=0}^\infty c_i x^i$ is a solution, for all $c_i\in \bbc^{N\times rN}$ satisfying the recursion relation
$$c_{j+1} = \sum_{i=0}^j\frac{1}{j+1}c_iU_{j-i}.$$
This establishes a $M_N(\bbc)$-linear map between the space of solutions $Y$ to $Y'=YU$ and the space of constant matrices $c_0\in \bbc^{N\times rN}$, which is a free $M_N(\bbc)$-module.  Under this isomorphism, the basis
$$(I,0,0,\dots, 0),\ \ (I,I,0,\dots, 0),\ \ (I,0,I,\dots, 0),\ \ (I,0,0,\dots, I),$$
corresponds to a $M_N(\bbc)$-module basis $\psi_0,\dots,\psi_r$ of $\ker(\mu)$.  By choice, each of the $\psi_i$ satisfies $\psi_i(0) = I$, and therefore $\psi_i$ is a unit in $M_N(\bbc[[x]])$ for all $i$.
\end{proof}

\begin{remk}
A point of caution is warrented here.  It is important in the above theorem that the coefficients of the operator reside in the power series ring $\bbc[[x]]$.  For example, the monic differential operator $\partial - \frac{1}{2x}$ has no kernel in $\bbc[[x]]$, or even in the ring of Laurent series $\bbc((x))$.
\end{remk}

\begin{remk}
The assumption that $\mu$ is monic is also important.  For example, if we consider the differential operator
$$\mu = \mxx{0}{1}{0}{0}\partial,$$
the kernel of $\mu$ is $\binom{\bbc[[x]]}{\bbc}$, which is not a finite-dimensional vector space.
\end{remk}

\begin{lem}
Suppose that $\mu$ is a monic matrix differential operator of order $r$, and that $\ker(\mu)$ has a $M_N(\bbc)$-module basis consisting of units Then $\mu$ has a factorization  of the form
$$\mu = (\partial I - \varphi_1)(\partial I - \varphi_2)\dots(\partial I - \varphi_r).$$
for some $\varphi_1,\dots,\varphi_r \in M_N(\bbc[[x]])$.
\end{lem}
\begin{proof}
We proceed by induction on $r$.  If $r=1$, then the statement is trivial.  Assume the result for all monic operators of order $r-1$, let $\mu$ be a monic operator of order $r$, and choose $\psi\in\ker(\mu)$ such that $\psi$ is a unit in $M_N(\bbc[[x]])$.  Writing
$$\psi\mu\psi^{-1} = \partial^rI + \sum_{i=0}^{r-1}\partial^i b_i(x)$$
for some $b_i(x)\in\bbc[[x]]$, the fact that $I\cdot(\psi\mu\psi^{-1}) = 0$ implies that $b_0(x) = 0$.  Therefore $\psi\mu\psi^{-1} = \partial \nu$ for some monic matrix differential operator $\nu$ of degree $r-1$.  Note that $\psi^{-1}\nu\psi$ is monic of order $r-1$, so by our inductive assumption, it has the desired factorization, and therefore
$$\mu = \psi^{-1}\partial\psi\psi^{-1}\nu\psi = (\partial I-\psi^{-1}\psi')\psi^{-1}\nu\psi$$
has the desired factorization also.
\end{proof}

The previous lemma allows us to provide necessary and sufficient conditions for a monic matrix differential operator to divide another matrix differential operator on the left.
\begin{lem}
Suppose that $\mu,\delta$ are matrix differential operators, with $\mu$ monic.  Then $\mu^{-1}\delta$ is a matrix differential operator if and only if $\ker(\mu)\subseteq\ker(\delta)$.
\end{lem}
\begin{proof}
From the previous lemma, it suffices to consider the case when $\mu$ is monic of order $1$.  In this case, $\ker(\mu)$ is generated as a left $M_n(\bbc)$-module by $\psi\in M_N(\bbc[[x]])$ with $\psi$ a unit in $M_N(\bbc[x])$.  Note that $\mu - (\partial I - \psi^{-1}\psi')$ is a differential operator of order $0$, and has kernel containing $\psi$.  Since $\psi$ is invertible in $M_N(\bbc[[x]])$ it follows that $\mu = \partial I - \psi^{-1}\psi'$.

Now since $\psi\in\ker(\delta)$, we have that $M_N(\bbc) I$ is in the kernel of $\psi\delta\psi^{-1}$.  It follows by the same argument as in the previous lemma that $\psi\delta\psi^{-1}=\partial \eta$ for some differential operator $\eta$.  Therefore
$$\delta = \psi^{-1}\partial\psi\psi^{-1}\eta\psi = (\partial - \psi^{-1}\psi')\psi^{-1}\eta\psi = \mu\psi^{-1}\eta\psi.$$
Therefore $\mu^{-1}\delta$ is a matrix differential operator.

Conversely, suppose $\mu := \nu^{-1}\delta$ is a matrix differential operator.  Then for any $\psi\in\ker(\nu)$ we have that
$$\psi\cdot\delta = \psi\cdot(\nu\mu) = (\psi\cdot\nu)\cdot\mu = 0I\cdot\mu = 0I.$$
Therefore $\psi\in\ker(\delta)$.  This completes the proof.
\end{proof}
We can loosen the restriction that $\mu$ is monic by instead assumin gthat there exists $\nu\in\mweyll$ such that $\nu\mu$ is monic.  This leads to the following generalization of the previous lemma.
\begin{lem}\label{kernel lemma 1}
Suppose that $\mu,\delta$ are matrix differential operators, and suppose that there exists a matrix differential operator $\nu$ such that $\nu\mu$ is monic.  Then $\mu^{-1}\delta$ is a matrix differential operator if and only if $\ker(\mu)\subseteq\ker(\delta)$.
\end{lem}
\begin{proof}
Clearly if $\mu^{-1}\delta$ is a matrix differential operator, then $\ker(\mu)\subseteq\ker(\delta)$.  Thus it suffices to prove the converse.  Suppose that $\ker(\mu)\subseteq\ker(\delta)$.  Then $\ker(\nu\mu)\subseteq\ker(\nu\delta)$, and therefore by the previous lemma, $\mu^{-1}\delta = (\nu\mu)^{-1}(\nu\delta) = \mu^{-1}\delta$ is a matrix differential operator.
\end{proof}

Instead of left-dividing, we are sometimes interested in conjugation.  The next lemma provides necessary and sufficient conditions for the conjugate to be a differential operator.
\begin{lem}\label{kernel lemma 2}
Suppose that $\mu,\delta$ are matrix differential operators, and suppose furthermore that there exists a matrix differential operator $\nu$ such that $\nu\mu$ is monic.  Then $\mu^{-1}\delta\mu$ is a matrix differential operator if and only if $\ker(\mu)\cdot\delta\subseteq\ker(\mu)$.
\end{lem}
\begin{proof}
From the previous lemma, $\mu^{-1}\delta\mu$ is a differential operator if and only if $\ker(\mu)\subseteq\ker(\delta\mu)$, which in turn holds if and only if $\ker(\mu)\cdot\delta\subseteq\ker(\mu)$.
\end{proof}

\begin{remk}
In the Terminology and Notation section above, we talk about $\mu^{-1}$, defined as a matrix pseudo-differential operator (if it exists).  By matrix pseudo-differential operator, we mean a pseudo-differential operator with coefficients in $M_N(\bbc((x)))$ for $\bbc((x))$ the field of Laurent series in $x$.  When does this pseudo-differential operator exist?  The answer is that $\mu^{-1}$ exists if and only if $\mu$ is not a zero-divisor in $\mweyll$.  In particular, this holds if the leading coefficient of $\mu$ has a determinant which is not identically zero.
\end{remk}

\begin{remk}
Suppose that we know $\nu^{-1}\delta\nu$ exists and is, in fact, some differential operator by the above lemmas.  How can we calculate its precise value?  One way is to note that $\wt\delta := \nu^{-1}\delta\nu$ must be a differential operator of order at most $\ell + 2m$ where $\ell$ is the order of $\delta$ and $m$ is the order of $\nu$.  The reason for this is that the order of $\nu^{-1}$ is bounded by the order of $\nu$.   Thus we can write $\wt \delta = \sum_{n=0}^{\ell+2m} \partial^n a_n(x)$ for some unknown functions $a_n(x)$ and then determine which values of the $a_n(x)$ work to satisfy the equation $\delta\nu = \nu\wt\delta$.  This removes the requirement of calculating $\nu^{-1}$ explicitly.  Note also that if $\nu$ has nonsingular leading coefficient, then the order of $\nu^{-1}\delta\nu$ is the same as the order of $\delta$, and therefore in such a case we may assume $a_n(x) = 0I$ for $n>\ell$.
\end{remk}

\subsection{The Proof}\label{The Proof Section}
Before proving the Main Theorem, we introduce one more lemma to help with some moment estimates.
\begin{lem}\label{trace lemma}
Suppose that $a,b\in M_N(\bbc)$ are positive-semidefinite.  Then
$$\tr(ab) \leq \tr(a)\tr(b).$$
\end{lem}
\begin{proof}
By Cauchy-Schwartz,
$$\tr(ab)\leq (\tr(a^2)\tr(b^2))^{1/2}.$$
Furthermore, since $a$ is positive-semidefinite,
$$\tr(a^2)\leq \tr(a)^2$$
and similarly for $b$.  From this the statement of the lemma follows immediately.
\end{proof}

We now have everything in place for the proof of the Main Theorem (Theorem (\ref{main theorem})).
\begin{proof}[Proof of Main Theorem]
\mbox{}
\begin{enumerate}[(a)]
\item  Note that since $\det(v(x))\in\bbc\diff\{0\}$, $v(x)^{-1}$ is also a polynomial.  Moreover since $w(x)$ is smooth and positive-definite on its support $(x_0,x_1)$, we may factor $w(x) = u(x)u(x)^*$ for some smooth function $u(x)$ of full rank on $(x_0,x_1)$.  Next, since $\delta$ is $w$-symmetric, the leading coefficient $a_2$ must be $w$-symmetric also.  Therefore $v^{-1}a_2(v^{-1})^\dag$ is also $w$-symmetric, and consequently $u^{-1}v^{-1}a_2(v^{-1})^\dag u$ is smooth and Hermitian on $(x_0,x_1)$.  Thus we may factor it as
$$u^{-1}v^{-1}a_2(v^{-1})^\dag u = hh^*$$
for some smooth function $h$ on $(x_0,x_1)$.  Taking $f = uhu^{-1}$, we have that
$$ff^\dag = uhu^{-1}w(uhu^{-1})^*w^{-1} = uhh^*u^{-1} = v^{-1}a_2 (v^{-1})^\dag$$
and therefore $f$ satisfies the properties stated in (a)

\item  From the definition of $f$, we have that $|\det(f)|^2 = |\det(h)|^2 = |\det(v)|^{-2}\det(a_2)\neq 0$ on $(x_0,x_1)$.  Therefore $f(x)$ has full rank for all $x\in (x_0,x_1)$, and it follows that $\wt w(x) = f(x)w(x)f(x)^*$ is positive-definite for all $x\in (x_0,x_1)$.  Moreover, $\wt w(x)$ is smooth on $(x_0,x_1)$, since it is the product of three smooth matrix-valued functions.  Therefore to prove that $\wt w(x)$ is a weight matrix, all that is left to show is that $\wt w(x)$ has finite moments.  Equivalently, we must show
$$\int_{x_0}^{x_1} |x|^n \tr(\wt w) dx < \infty.$$
To show this, first note that
$$\tr(\wt w) = \tr(fuu^*f^*) = \tr(uhh^*u^*) = \tr(hh^*u^*u).$$
By Lemma (\ref{trace lemma})
\begin{align*}
\tr(hh^*u^*u)
  & \leq \tr(hh^*)\tr(u^*u) = \tr(u^{-1}v^{-1}a_2(v^{-1})^\dag u)\tr(w)\\
  & = \tr((v^{-1})^{\dag}a_2v^{-1})\tr(w)
\end{align*}
and therefore
$$\int_{x_0}^{x_1} |x|^n \tr(\wt w) dx < \int_{x_0}^{x_1} |x|^n \tr(v^{-1}(x)a_2(x)(v^{-1}(x))^\dag)\tr(w(x))dx < \infty$$
by H\"{o}lder's inequality and the fact that $w(x)$ has finite moments.  Therefore $\wt w$ is a weight matrix.

\item
We calculate
\begin{align*}
\mu
  & = -v_1^{-1}a_2(v_1^{-1})^\dag\nu^\dag\\
  & = -v_1^{-1}a_2(\partial -v_0v_1^{-1})^\dag\\
  & = -v_1^{-1}a_2(-\partial - w'w^{-1} - (v_0v_1^{-1})^\dag)\\
  & = \partial v_1^{-1}a_2 + v_1^{-1}a_2w'w^{-1} + v_1^{-1}a_2(v_0v_1^{-1})^\dag + (v_1^{-1}a_2)'
\end{align*}

Furthermore, the noncommutative Pearson equation tells us that
$$(v_1^{-1}a_2)' = (v_1^{-1})'a_2 + v_1^{-1}\frac{1}{2}(a_1+a_1^\dag) - v_1^{-1}a_2w'w^{-1}.$$
Substituting this in and simplifying, we find
$$\mu = \partial v_1^{-1}a_2 - v_0^{-1}a_0 + v_0^{-1}[(v_0v_1^{-1}a_2(v_0v_1^{-1})^\dag - (v_0v_1^{-1})^2a_2]$$
Then using the fact that $(v_0v_1^{-1})^\dag = v_0v_1^{-1}$ the last summand cancels out, leaving
$$\mu = \partial v_1^{-1}a_2 - v_0^{-1}a_0$$
Using this, we calculate
\begin{align*}
\nu\mu
  & = (\partial v_1 - v_0)(\partial v_1^{-1}a_2 - v_0^{-1}a_0)\\
  & = \partial^2 a_2 + \partial [(v_1)'v_1^{-1}a_2-v_0v_1^{-1}a_2 - v_1v_0^{-1}a_0] + a_0\\
  & = \partial^2 a_2 + \partial [-(v_0v_1^{-1})^{-1}(v_0v_1^{-1})'a_2-v_0v_1^{-1}a_2 - v_1v_0^{-1}a_0] + a_0\\
  & = \partial^2 a_2 + \partial -(v_0v_1^{-1})^{-1}[(v_0v_1^{-1})'a_2+(v_0v_1^{-1})^2a_2 + a_0] + a_0\\
  & = \partial^2 a_2 + \partial a_1 + a_0 = \delta.
\end{align*}

Next note that $\nu$ is degree-filtration preserving.  If $\nu$ is not degree preserving, then there exists a nonzero polynomial $p(x)$ satisfying $p(x)\cdot\nu = 0$.  Then
$$\deg(p) = \deg(p\cdot\delta) = \deg(p\cdot\nu\cdot\nu^{-1}\delta) = \deg(0) = -\infty.$$
This is a contradiction, and therefore $\nu$ is degree-preserving.  Since $\delta = \nu\mu$ and $\nu$ are degree-preserving, it follows that $\mu$ is also degree-preserving.

\item The assumption that $\lim_{x\rightarrow x_i} (|x|^n+1)w(x) = 0$ implies that $\nu f$ is $w$-adjointable by Lemma (\ref{first order adjoint}).  Moreover, we calculate:
\begin{align*}
\langle \wt p(x,m), \wt p(x,n)\rangle_{\wt w}
  & = \langle p(x,m)\cdot\nu f, p(x,n) \cdot \nu f\rangle_w = \langle p(x,m),p(x,n)\cdot \nu ff^\dag\nu^\dag\rangle_w\\
  & = \langle p(x,m),p(x,n)\cdot(\nu\mu)\rangle_w = \langle p(x,m),p(x,n)\cdot\delta\rangle_w
\end{align*}
Since $\delta$ is degree-filtration preserving, if $m>n$ then $p(x,n)\cdot(\delta + \nu c)$ may be written as a $M_N(\bbc)$-linear combination of $p(x,0),\dots p(x,n)$, and therefore $\langle p(x,m),p(x,n)\cdot(\delta + \nu c)\rangle_w = 0$.  Consequently $\langle \wt p(x,m), \wt p(x,n)\rangle_{\wt w} = 0I$ for $m>n$.  If $m<n$, then
$$\langle \wt p(x,m),\wt p(x,n)\rangle_{\wt w} = \langle \wt p(x,n),\wt p(x,m)\rangle_{\wt w}^* = (0I)^* = 0I.$$
Therefore in any event
$$\langle \wt p(x,m),\wt p(x,n)\rangle_{\wt w} = 0I,\ \ m\neq n.$$
Moreover $\wt p(x,m)$ is a polynomial of degree $m$ for each $m$, since $\nu$ is degree-preserving.  It follows that $\wt p(x,m)$ is a sequence of orthogonal matrix polynomials for $\wt w$.

\item
We know that $p(x,n)\cdot\delta = \lambda_n p(x,n)$ for some $\lambda_n\in M_N(\bbc)$ for all $n\geq 0$.  We calculate
$$\wt p(x,n)\cdot \wt\delta = p(x,n)\cdot (\nu\mu\nu) = p(x,n)\cdot (\delta\nu) = \lambda_n p(x,n)\cdot\nu = \lambda_n\wt p(x,n).$$
Since $\wt p(x,n)$ is a sequence of orthogonal matrix polynomials for $\wt w$, this shows that $\wt\delta\in D(\wt w)$.  To complete the proof, we must show that $\wt\delta$ is $\wt w$-symmetric.  Since $\wt\delta\in D(\wt w)$, we know that $\wt\delta$ is $\wt w$-adjointable, so it suffices to prove that $\wt\delta$ is formally $\wt w$-symmetric, ie. that
$$\wt w (\wt\delta)^* (\wt w)^{-1} = \wt \delta.$$
We calculate
\begin{align*}
\wt w (\wt\delta)^* (\wt w)^{-1}
  & = fwf^*(ff^\dag\nu^\dag\nu)^* (f^*)^{-1}w^{-1}f^{-1}\\
  & = fwf^*\nu^*(\nu^\dag)^*(f^\dag)^*f^*(f^*)^{-1}w^{-1}f^{-1}\\
  & = fwf^*\nu^*(\nu^\dag)^*(f^\dag)^*w^{-1}f^{-1} = fw(\nu f)^*(\nu f)^{\dag *}w^{-1}f^{-1}\\
  & = fw((\nu f)^\dag(\nu f))^*w^{-1}f^{-1} = f((\nu f)^\dag(\nu f))^\dag f^{-1}\\
  & = f(\nu f)^\dag(\nu f) f^{-1} = ff^\dag\nu^\dag\nu = \wt\delta.
\end{align*}
This proves (f).

\item
Suppose that $\wt\delta'\in (\nu^{-1} D(w)\nu)\cap M_N(\Omega)$.  Then there exists $\delta'\in D(w)$ satisfying $\delta'\nu = \nu\wt\delta'$.  Moreover, there exists a sequence $\lambda_0',\lambda_1',\dots\in M_N(\bbc)$ satisfying $p(x,n)\cdot\delta' = \lambda_n p(x,n)$ for all $n\geq 0$.  Therefore we have that
$$\wt p(x,n)\cdot\wt\delta' = p(x,n)\cdot\nu\wt\delta' = p(x,n)\cdot(\delta'\nu) = \lambda_n' p(x,n)\cdot\nu = \lambda_n'\wt p(x,n).$$
This shows that $\delta'\in D(\wt w)$, proving that $(\nu^{-1} D(w)\nu)\cap M_N(\Omega)\subseteq D(\wt w)$.  The fact that $(\nu^{-1} D(w)\nu)\cap M_N(\Omega) = \{\nu^{-1}\delta\nu: \delta\in D(w),\ \ker(\nu)\cdot\delta\subseteq\ker(\nu)\}$ is just a restatement of Lemma (\ref{kernel lemma 2}).
\end{enumerate}
\end{proof}

We next prove Proposition \ref{subalgebra proposition}.  Before doing so, we establish the following lemma.
\begin{lem}
Let $\nu$ be a Darboux transformation from a Bochner pair $(w,\delta)$ to a Bochner pair $(\wt w,\wt\delta)$.  Then $D(\wt w)\wt\delta \subseteq D(\wt w,\nu,w)$.
\end{lem}
\begin{proof}
Suppose that $\wt\eta\in D(\wt w)$.  Then there exists $\mu\in\mweyll$ satisfying $\delta = \nu\mu$ and $\wt\delta = \mu\nu$ and therefore
$$\eta := \nu(\wt\eta\wt\delta)\nu^{-1} = \nu\wt\eta\mu\in\mweyll.$$
Let $\{p(x,n)\}$ be a sequence of orthogonal matrix polynomials for $w$.  Then there exists a sequence $\lambda_0,\lambda_1,\dots\in M_N(\bbc)$ satisfying $p(x,n)\cdot\delta = \lambda_n p(x,n)$ for all $n\geq 0$.  Also $\wt p(x,n) := p(x,n)\cdot\nu$ defines a sequence of orthogonal matrix polynomials for $\wt w$, and therefore there exists a sequence $\wt\lambda_0,\wt\lambda_1,\dots\in M_N(\bbc)$ satisfying $\wt \lambda_n\wt p(x,n) = \wt p(x,n)\cdot\wt\eta$.  We calculate
\begin{align*}
p(x,n)\cdot\eta
  & = p(x,n)\cdot(\nu\wt\eta\mu) = \wt p(x,n)\cdot(\wt \eta \mu)\\
  & = \wt\lambda_n \wt p(x,n)\cdot \mu = \wt\lambda_n p(x,n)\cdot(\nu\mu)\\
  & = \wt\lambda_n p(x,n)\cdot\delta = \wt\lambda_n\lambda_n p(x,n).
\end{align*}
Therefore $\eta\in D(w)$.  Furthermore $\nu^{-1}\eta\nu = \wt\eta\wt\delta$, and it follows that $\wt\eta\wt\delta\in D(\wt w,\nu, w)$.  Since $\wt\eta\in D(\wt w)$ was arbitrary, this proves our Lemma.
\end{proof}

\begin{proof}[Proof of Proposition \ref{subalgebra proposition}]
Suppose that $\wt \delta$ and $\ell$ satisfy the assumptions of the statement of the proposition.  Then for all $\wt \eta\in\mweyl$, the order of $\wt \eta\wt\delta$ is the order of $\wt \eta$ plus two.

By the previous lemma, $D(\wt w)\wt\delta \subseteq D(\wt w,\nu,w)$.  Therefore multiplication by $\wt \delta$ defines a $\bbc$-linear monomorphism $D(\wt w)_i\rightarrow D(\wt w,\nu,w)_{i+2}$.  This in turn restricts to a monomorphism
$$D(\wt w)_i/D(\wt w)_{i-1}\rightarrow D(\wt w,\nu,w)_{i+2}/D(\wt w,\nu,w)_{i+1}.$$
Furthermore, the inclusion $D(\wt w,\nu,w)\subseteq D(\wt w)$ induces an injection
$$D(\wt w,\nu,w)_i/D(\wt w,\nu,w)_{i-1}\rightarrow D(\wt w)_i/D(\wt w)_{i-1}$$
and therefore
$$\dim_\bbc\left(\frac{D(\wt w,\nu,w)_i}{D(\wt w,\nu,w)_{i-1}}\right)\leq \dim_\bbc\left( \frac{ D(\wt w)_i}{D(\wt w)_{i-1}}\right)\leq \dim_\bbc\left(\frac{D(\wt w,\nu,w)_{i+2}}{D(\wt w,\nu,w)_{i+1}}\right).$$
Thus for $i\geq \ell$ all of the above dimensions are equal, and therefore
$$D(\wt w,\nu,w)_i/D(\wt w,\nu,w)_{i-1} \xrightarrow{\cong} D(\wt w)_i/D(\wt w)_{i-1},\ \forall i\geq\ell.$$
Hence $D(\wt w)$ is generated over $D(\wt w,\nu,w)$ by elements of order $< \ell$.
\end{proof}

\section{Explicit Examples}\label{example section}
In this section we provide explicit examples of the Main Theorem in action.  The method of finding examples is straightforward: we consider a specific Bochner pair $(w,\delta)$ satisfying the assumptions of the Main Theorem, and then attempt to find a degree $1$ matrix polynomial $v_1(x)$ and a constant matrix $v_0$ satisfying the differential equation presented in the Main Theorem.  Once we find such a pair $v_0,v_1(x)$, all that is left is to check that it satisfies the remaining assumptions of the Main Theorem.  One helpful thing to point out in our search for $v_0,v_1(x)$ is that if $v_1(x)$ satisfies the assumptions of the Main Theorem, then $v_0v_1(x)^{-1}a_2(x)$ will also be a polynomial of degree one.  Therefore it makes sense to restrict our search to polynomials of the form
$$v_0v_1(x)^{-1} = (A_1x + A_0)a_2(x)^{-1}$$
such that $v_1(x)$ is also a polynomial of degree $1$.  Inserting this into the differential equation in the Main Theorem then results in a pair of algebraic relations that $A_0$ and $A_1$ must satisfy.  To illustrate this point further, we provide the following two examples.
\begin{ex}
Consider a Bochner pair of the form $(e^{-x^2},\delta)$ for $\delta = -\partial^2I + \partial 2xI + B$ for some Hermitian matrix $B\in M_N(\bbc)$.  Any such matrix $B$ will still give a Bochner pair, so we will leave it as flexible for now.  Suppose that $v_0v(x)^{-1} = A_1x+ A_0$ satisfies the differential equation of the Main Theorem.  Plugging this in and comparing coefficients of similar powers of $x$, we obtain the relations:
$$(A_1-2I)A_1 = 0I,\ \ A_1A_0 + A_0A_1 - 2A_0 = 0I,\ \ A_0^2 + A_1 = B.$$
Choosing $A_0,A_1,$ and $B$ so as to satisfy these equations, we obtain a potential candidate for $v(x)$.  Below we explore one family of solutions of these equations.
\end{ex}
\begin{ex}
Consider a bochner pair of the form $(x^be^{-x}1_{(0,\infty)},\delta)$ for $\delta = -\partial^2xI + \partial(x-(b+1))I + B$ for some Hermitian matrix $B\in M_N(\bbc)$.  Again, any such $B$ will give a Bochner pair, so we may use $B$ as an additional variable.  Suppose that $v_0v(x)^{-1} = (A_1x + A_0)(1-x^2)^{-1}x^{-1}$ satisfies the differential equation of the Main Theorem.  Plugging this in and comparing coefficients of similar powers of $x$, we obtain the relations:
$$A_1(A_1-I) = 0I,\ \ A_0(A_0 + bI) = 0I,\ \ B = A_1A_0 + A_0A_1 + bA_1 - A_0.$$
Choosing $A_0,A_1,$ and $B$ so as to satisfy these equations, we obtain a potential candidate for $v(x)$.  We do not explore this further in the examples.
\end{ex}
\begin{ex}
Consider a Bochner pair of the form $((1-x^2)^{r/2},\delta)$ for $\delta = -\partial^2(1-x^2)I + \partial x(r+2)I + B$ for some Hermitian matrix $B\in M_N(\bbc)$.  Again, any such matrix $B$ will give a Bochner pair, so we may use $B$ as an additional variable.  Suppose that $v_0v(x)^{-1} = (A_1x+ A_0)(1-x^2)^{-1}$ satisfies the differential equation of the Main Theorem.  Plugging this in and comparing coefficients of similar powers of $x$, we obtain the relations:
$$A_1^2 - rA_1 + A_0^2 = 0I,\ \ A_1A_0 + A_0A_1 - rA_0 = 0I,\ \ B = A_1 + A_0^2.$$
Choosing $A_0,A_1,$ and $B$ so as to satisfy these equations, we obtain a potential candidate for $v(x)$.  Below we explore one family of solutions of these equations.
\end{ex}
The relations in each of these examples are the central calculation which generated the examples provided below.  Many of the results in the example section below were double-checked using python code with the Sympy symbolic computation library\cite{sympymanual}.
\subsection{A Family of Examples of Hermite Type}\label{hermite example section}
As our first example, we consider Darboux transformations of Bochner pairs $(w,\delta)$ for
$$\delta = -\partial^2I + \partial2xI + B,\ \ w = e^{-x^2}I,\ \ B = \mxx{CC^* + 2I}{0I}{0I}{C^*C}.$$
for some nonsingular, normal $C\in M_{N/2}(\bbc)$ (we need $C$ to be normal so that $\delta$ is $w$-symmetric).  For this to work, we must assume $N$ is even.  According to the Main Theorem, we can find examples by finding an invertible, degree $1$ polynomial $v(x)$ satisfying a certain differential equation, along with a handful of other properties.  Define $v_1(x)$ and $v_0(x)$ by
$$v(x)^{-1} = \mxx{2I}{0I}{0I}{0I}x + \mxx{0I}{C}{C^*}{0I},\ \ v_0(x) = I.$$
Then $v_0(x),v_1(x)$ satisfy the requirements stated in Theorem 1.0.3. for the Bochner pair $(w,\delta)$.
\begin{prop}
The functions $v_0(x),v_1(x)$ as defined above satisfies the assumptions of the Main Theorem for the Bochner pair $(e^{-x^2}I,\delta)$ also defined above.
\end{prop}
\begin{proof}
We calculate that $(v_1(x)^{-2}-2Ixv_1^{-1}) = \mxx{CC^*}{0I}{0I}{C^*C}$, and therefore
$$v_1(x) = \mxx{CC^*}{0I}{0I}{C^*C}^{-1}(v_1(x)^{-1} - 2xI) = \mxx{0I}{0I}{0I}{-2(C^*C)^{-1}}x + \mxx{0I}{(C^*)^{-1}}{C^{-1}}{0I}.$$
In particular $v_1(x)$ is a polynomial of degree $1$ and $\det(v(x))\in\bbc\diff\{0\}$ (because $v_1(x)$ is a unit in $M_N(\bbc[x])$).  We also see that $v_1(x)^* = v_1(x)$ and that
$$\tr((v(x)v(x)^\dag)^{-1}) = \tr(v(x)^{-2}) = 2\ell x + 2\tr(CC^*)$$
is also a polynomial and therefore in $L^2(\tr w(x))$.  Furthermore, we calculate
$$v(x)^{-2}a_2(x) + (v(x)^{-1})'a_2(x) + v(x)^{-1}a_1(x) + a_0(x) = 0$$
for $a_0(x) = B$.  This completes the proof.
\end{proof}
Thus as a consequence of the main theorem $\nu = \partial v_1(x) - v_0(x)$ defines a Darboux transformation from the Bochner pair $(e^{-x^2}I,\delta)$ to the Bochner pair $(\wt w, \wt \delta)$, where
$$\wt w = e^{-x^2}\mxx{4x^2I + CC^*}{2xC}{2xC^*}{C^*C},\ \ \wt\delta = -\partial^2I + \partial\left(xI - \mxx{0I}{-4I(C^*)^{-1}}{0I}{0I}\right) + B.$$
The Main Theorem also tells us a sequence of orthogonal matrix polynomials for $\wt w$, and this allows us to determine a generating function for them.
\begin{prop}
The monic orthogonal matrix polynomials $\wt p(x,n)$ for the weight matrix $\wt w$ define above are given by the generating function formula
$$\sum_{n=0}^\infty \mxx{0I}{(C^*)^{-1}}{C^{-1}}{-2(C^*C)^{-1}n}\frac{t^n}{n!}\wt p(x,n) = e^{xt-t^2/4}\mxx{-I}{(C^*)^{-1}t}{C^{-1}t}{-2(C^*C)^{-1}xt-I},$$
\end{prop}
\begin{proof}
The (monic) Hermite polynomials $h_n(x)I$ define a sequence of monic orthogonal polynomials for $e^{-x^2}I$, and satisfy the generating function formula
$$\sum_{n=0}^\infty \frac{t^n}{n!}h_n(x)I = \exp(xt - t^2/4)I.$$
By the Main Theorem, $h_n(x)I\cdot\nu$ defines a sequence of orthogonal matrix polynomials for the weight matrix $\wt w$.  Since the leading coefficient of $h_n(x)I$ is $I$, it's easy to see that the leading coefficient of $h_n(x)I\cdot\nu$ is
$$c_n = \mxx{0}{(C^*)^{-1}}{C^{-1}}{-2(C^*C)^{-1}n}.$$
Thus $\wt p(x,n) := c_n^{-1}h_nI\cdot\nu$ is the monic sequence of orthogonal matrix polynomials for $\wt w$.  Multiplying both sides of the generating function formula for the Hermite polynomials by $\nu$ and using this expression for $\wt p(x,n)$ then results in the generating function formula in the statement of the theorem.
\end{proof}

For $N=2$ and $C=c\in\bbc\diff\{0\}$, this weight matrix appears first in \cite{duran2004orthogonal} and later in \cite{castro2006}, where explicit generators and relations of the associated algebra $D(\wt w)$ are listed without proof.  This set of generators and relations is verified in \cite{tirao2011} in a $30$-page tour-de-force.  The center of $D(\wt w)$ is also determined explicitly, though misidentified as being an elliptic curve, rather than a singular cubic plane curve.  In the following, we demonstrate the utility of the main theorem by rederiving the structure of $D(\wt w)$ and the center of $D(\wt w)$ succinctly.  Better yet, we show that $D(\wt w)$ is naturally identified with a certain subalgebra of $2\times 2$ matrices over a polynomial ring, as shown in Equation (\ref{first example eqn}).  The Main Theorem combined with Proposition (\ref{subalgebra proposition}) gives us a means to calculate the structure of the algebra $D(\wt w)$ associated with the weight matrix $\wt w$.  We will use this to calculate the structure of the algebra $D(\wt w)$ in the case that $N=2$ and $C=c\in \bbc\diff\{0\}$.
\begin{prop}\label{hermite example presentation}
Suppose $N=2$ and $C = c\in \bbc\diff\{0\}$.  The algebra $D(\wt w)$ is given by
\begin{equation}\label{first example eqn}
D(\wt w) = \left\lbrace\nu^{-1}\left(\begin{array}{cc}
f_{11}(\epsilon) & f_{12}(\epsilon)\\
f_{21}(\epsilon) & f_{22}(\epsilon)
\end{array}\right)\nu: f_{ij}(z)\in \bbc[z],\ \substack{f_{12}(|c|^2+2) = 0,\ f_{21}(|c|^2) = 0,\\ f_{11}(|c|^2+2) = f_{22}(|c|^2)}\right\rbrace.
\end{equation}
for $\epsilon = \partial^2-2\partial x$.
\end{prop}
\begin{proof}
Note that $D(w) = M_N(\bbc[\epsilon])$ for $\epsilon = \partial^2-2\partial x$ the Hermite operator.  Thus 
$$D(\wt w) \supseteq \{\nu^{-1}f(\epsilon)\nu: f(\epsilon)\in M_N(\bbc[\epsilon]),\ \ker(\nu)\cdot f(\epsilon)\subseteq \ker(\nu)\}.$$
If $\psi_1$ and $\psi_2$ form a basis for the eigenfunctions of $\epsilon$ with eigenvalue $|b|^2$, then $\psi_1'$ and $\psi_2'$ form a basis for the eigenfunctions of $\epsilon$ with eigenvalue $|b|^2+2$, and
$$\psi = \left(\begin{array}{cc}
\psi_1' & b\psi_1\\
\psi_2' & b\psi_2
\end{array}\right)$$
is a generator for the kernel of $\nu$.  Therefore
$$D(\wt w)\supseteq D(\wt w,\nu,w) = \{\nu^{-1} f(\epsilon)\nu: f(\epsilon)\in M_N(\bbc[\epsilon]),\ \psi\cdot f(\epsilon)\in M_N(\bbc)\psi\}.$$
Now consider $f(\epsilon) = (f_{ij}(\epsilon))\in M_2(\bbc[\epsilon])$.  An elementary argument using the fact that $\psi_1,\psi_2,\psi_1',\psi_2'$ must all be linearly independent, shows us that the only way that $\psi$ can be an eigenfunction of $f(\epsilon)$ is if $f_{12}(|b|^2+2) = 0$, $f_{21}(|b|^2) = 0$, and $f_{11}(|b|^2+2)=f_{22}(|b|^2)$.  Therefore
$$
D(\wt w,\nu,w) = \left\lbrace\nu^{-1}\left(\begin{array}{cc}
f_{11}(\epsilon) & f_{12}(\epsilon)\\
f_{21}(\epsilon) & f_{22}(\epsilon)
\end{array}\right)\nu: f_{ij}(z)\in \bbc[z],\ \substack{f_{12}(|b|^2+2) = 0,\ f_{21}(|b|^2) = 0,\\ f_{11}(|b|^2+2) = f_{22}(|b|^2)}\right\rbrace.$$

To finish our proof, we must show that the inclusion $D(\wt w,\nu,w)\subseteq D(\wt w)$ is in fact an equality.  Since the leading coefficient of $\nu$ is nonsingular, if $\eta\in\mweyll$ is such that $\nu^{-1}\eta\nu$ is a differential operator, then the order of $\nu^{-1}\eta\nu$ is the same as the order of $\eta$.  Furthermore the order of $\mxx{f_{11}(\epsilon)}{f_{12}(\epsilon)}{f_{21}(\epsilon)}{f_{22}(\epsilon)}$ is equal to $2\max_{ij}\deg(f_{ij})$.  Looking at our conditions for the $f_{ij}$'s, we see that
$$D(\wt w,\nu, w) = \bbc I \oplus \nu^{-1}\mxx{p(\epsilon-2)\bbc}{p(\epsilon-2)\bbc}{p(\epsilon)\bbc}{p(\epsilon)\bbc}\nu \oplus \nu^{-1}M_2(\bbc[\epsilon]p(\epsilon)p(\epsilon-2))\nu,$$
for $p(\epsilon) = \epsilon-|b|^2$.
This shows in particular that $D(\wt w,\nu,w)$ has a one-dimensional subspace of operators of order $0$, no operators of odd order, and a four-dimensional subspace of operators of order $2m$ for every integer $m>0$.  In other words
$$\dim(D(\wt w,\nu, w)_i/D(\wt w,\nu, w)_{i-1}) = \left\lbrace\begin{array}{cc}
4, & \text{if $i>0$ is even}\\
0, & \text{if $i>0$ is odd}
\end{array}\right.$$
Therefore by Proposition \ref{subalgebra proposition}, $D(\wt w)$ is generated over $D(\wt w,\nu,w)$ by operators of order $0$.  However, the only operators of order $0$ in $D(\wt w)$ are those in $\bbc I$, and therefore $D(\wt w) = D(\wt w,\nu,w)$.
\end{proof}

The proof of the previous proposition shows us that $D(\wt w)$ is given by
$$D(\wt w) = \bbc I \oplus \nu^{-1}\mxx{p(\epsilon-2)\bbc}{p(\epsilon-2)\bbc}{p(\epsilon)\bbc}{p(\epsilon)\bbc}\nu \oplus \nu^{-1}M_2(\bbc[\epsilon]p(\epsilon)p(\epsilon-2))\nu,$$
A quick check shows that this algebra is generated by the four elements of order two.  Moreover, since the center of $M_2(\bbc[\epsilon]p(\epsilon)p(\epsilon-2))$ is $\bbc[\epsilon]p(\epsilon)p(\epsilon-2)I$, we have that the center of $D(\wt w)$ is
$$Z(\wt w) = \bbc I + \nu^{-1}\bbc[\epsilon]p(\epsilon)p(\epsilon-2)I\nu.$$
This is a commutative algebra whose spectrum is isomorphic to the nodal cubic curve $y^2 - 2(|b|^2+1)xy + |b|^2(|b|^2+2)x^2 - x^3 = 0$, via the isomorphism induced by the ring isomorphism defined by
$$\bbc[x,y]\rightarrow \bbc I + \nu^{-1}\bbc[\epsilon]p(\epsilon)p(\epsilon-2)I\nu,\ \ \ x\mapsto \nu^{-1} p(\epsilon)p(\epsilon-2)I\nu,\ y\mapsto \nu^{-1}\epsilon p(\epsilon)p(\epsilon-2)I\nu.$$
The ring $D(\wt w)$ has the structure of a module over its center, and as a module it is finitely generated and torsion-free.  However, it is not a free module over its center since
$$\mxx{0}{0}{p(\epsilon)}{0}(p(\epsilon)p(\epsilon-2)^2I) + \mxx{0}{0}{p(\epsilon)p(\epsilon-2)}{0}(-p(\epsilon)p(\epsilon-2)I) = 0I.$$
This completes our example.

\subsection{A Family of Examples of Jacobi Type}\label{jacobi example section}
Next, we consider Darboux transformations of Bochner pairs $(w,\delta)$ for
$$\delta = -\partial^2(1-x^2)I + \partial(r+2)xI + B,\ \ w = (1-x^2)^{r/2}1_{(-1,1)}I$$
where here
$$B = \mxx{S^2(I+T^2)}{0I}{0I}{T^2(I+S^2)}.$$
for some nonsingular, Hermitian $S,T\in M_{N/2}(\bbc)$ satisfying $S^2+T^2 = rI$, and for $r>2$.  Note that for this to work, we must again assume $N$ is even.  According to the Main Theorem, we can find examples by finding an invertible, degree $1$ polynomial $v(x)$ satisfying a certain differential equation, along with a handful of other properties.  Define $v_1(x)$ and $v_0(x)$ by
$$v_1(x)^{-1} = \left(\mxx{S}{0I}{0I}{T}x + \mxx{0I}{-S}{-T}{0I}\right)(1-x^2)^{-1},\ \ v_0(x) = \mxx{S}{0}{0}{T}.$$
Then $v_1(x),v_0(x)$ satisfy the requirements stated in Theorem 1.0.3. for the Bochner pair $(w,\delta)$
\begin{prop}
The functions $v_0(x),v_1(x)$ as defined above satisfy the assumptions of the Main Theorem for the Bochner pair $(w,\delta)$ also defined above.
\end{prop}
\begin{proof}
We calculate
$$v_1(x) = \mxx{-S^{-1}}{0I}{0I}{-T^{-1}}x + \mxx{0I}{-T^{-1}}{-S^{-1}}{0I}.$$
In particular $v_i(x)$ is a polynomial of degree $i$ for $i=0,1$ and $\det(v_i)$ is nonzero on $(-1,1)$.  We also see that $(v_i(x)^{-1})^\dag = (v_i(x)^{-1})^*$ and that
$$\tr((v_1(x)a_2(x)v_1(x)^\dag)^{-1}) = \tr(v_1(x)^{-1}a_2(x)(v_1^{-1})^*) = \tr(SS^* + TT^*)\frac{1+x^2}{1-x^2}$$
which is in $\in L^2(\tr(w(x))dx)$ since $r>2$.  Moreover, $v_0v_1^{-1}$ is Hermitian and therefore $(v_0v_1^{-1})^\dag = (v_0v_1^{-1})^* = v_0v_1^{-1}$.
Lastly, we calculate
$$v(x)^{-2}a_2(x) + (v(x)^{-1})'a_2(x) + v(x)^{-1}a_1(x) + a_0(x) = 0$$
for $a_0(x) = B$.  This completes the proof.
\end{proof}

Thus as a consequence of the main theorem $\nu = \partial v_1(x) - v_0(x)$ defines a Darboux transformation from the Bochner pair $(w,\delta)$ to the Bochner pair $(\wt w, \wt \delta)$, where
$$\wt w = (1-x^2)^{r/2-1}\mxx{S^2x^2 + T^2}{-(S^2+T^2)x}{-(S^2+T^2)x}{T^2x^2 + S^2}$$
The Main Theorem also tells us a sequence of orthogonal matrix polynomials for $\wt w$, and this allows us to determine a generating function for them, consistent with the generating function found in \cite{duran2005structural} for $N=2$.
\begin{prop}
A sequence of (non-monic) orthogonal matrix polynomials $\wt p(x,n)$ for the weight matrix $\wt w$ define above are given by the generating function formula
$$\sum_{n=0}^\infty \wt p_n(x)t^n = \psi_x(x,t)v_1(x) + \psi(x,t)v_0(x),$$
where $\psi(x,t) = 2^{r-1}\phi(x,t)^{-1}((1+\phi(x,t))^2-t^2)^{(1-r)/2}$ and $\phi(x,t) = (1-2xt + t^2)^{1/2}$.
\end{prop}
\begin{proof}
The Gegenbauer polynomials $h_n(x)I$ define a sequence of monic orthogonal polynomials for $w(x)$, and satisfy the generating function formula
$$\sum_{n=0}^\infty h_n(x)t^nI = \psi(x,t)I.$$
By the Main Theorem, $h_n(x)I\cdot\nu$ defines a sequence of orthogonal matrix polynomials for the weight matrix $\wt w$.  Multiplying both sides of the generating function formula for the Gegenbauer polynomials by $\nu$ leads to the desired generating function formula.
\end{proof}

For $N=2$, $S^2=p$, and $T^2=r-p$, this weight matrix appears in \cite{zurrian2015}\cite{zurrian2016algebra}, where explicit generators and relations of the associated algebra $D(\wt w)$ are computed, though the effort involved is again substantial.  In this example, we verify the calculation of the algebra $D(\wt w)$ calculated in the paper, using the framework established above.  As with Tirao's example (eg. the example of Section \ref{hermite example section} with $N=2$), the machinery gives the algebra $D(\wt w)$ with significantly less wrangling.  In particular, the Main Theorem combined with Proposition (\ref{subalgebra proposition}) gives us a means to calculate the structure of the algebra $D(\wt w)$ associated with the weight matrix $\wt w$, which we do in this special case.
\begin{prop}\label{jacobi example presentation}
Suppose $N=2$ and $S^2=p$ and $T^2=r-p$ for $p\in (0,r)$.  The algebra $D(\wt w)$ is given by
$$D(\wt w) = D(\wt w,\nu,w)= \nu^{-1}\left\lbrace \left(\begin{array}{cc}
f_{11}(\epsilon) & f_{12}(\epsilon)\\
f_{21}(\epsilon) & f_{22}(\epsilon)
\end{array}\right) : \substack{f_{12}(p(n-p+1)) = 0,\ f_{21}((n-p)(p+1)) = 0,\\ f_{11}(p(n-p+1)) = f_{22}((n-p)(p+1))}\right\rbrace \nu.$$
for $\epsilon = -\partial^2(1-x^2) + \partial x(r+2)$.
\end{prop}
\begin{proof}
Let $\epsilon = -\partial^2(1-x^2) + \partial x(r+2)$.  Then $D(w) = M_2(\bbc[\epsilon])$.  Let $\psi_1,\psi_2$ be eigenfunctions of $\epsilon$ with eigenvalues $p(r-p+1)$ and $(r-p)(p+1)$, respectively.  Then $\phi_2 = \frac{x^2-1}{n-p}\psi_1' + \frac{p}{r-p}x\psi_1$ is an eigenfunciton of $\epsilon$ with eigenvalue $(r-p)(p+1)$, and $\phi_1 = \frac{x^2-1}{p}\psi_2' + \frac{r-p}{p}x\psi_2$ is an eigenfunction of $\epsilon$ with eigenvalue $p(r-p+1)$.  Moreover the kernel of $\nu$ is
$$\ker(\nu) = M_2(\bbc)\psi(x),\ \ \psi(x) = \left(\begin{array}{cc}
\psi_1 & \phi_2\\
\phi_1 & \psi_2.
\end{array}\right).$$
Our criterion for $\nu^{-1}\eta\nu$ to be a differential operator then implies that
$$\nu D(\wt w, \nu, w)\nu^{-1} = \{f(\epsilon)\in M_2(\bbc[\epsilon]): \psi\cdot f(\epsilon)\in M_2(\bbc)\psi\}.$$
The functions $\psi_1,\psi_2,\phi_1,\phi_2$ are linearly independent.  Suppose that
$$f(\epsilon) = \left(\begin{array}{cc}
f_{11}(\epsilon) & f_{12}(\epsilon)\\
f_{21}(\epsilon) & f_{22}(\epsilon)
\end{array}\right)\in M_2(\bbc[\epsilon])$$
satisfies
$$\psi\cdot f(\epsilon) = \lambda\psi,\ \ \lambda = \left(\begin{array}{cc}a & b\\ c & d\end{array}\right)\in M_2(\bbc).$$
Using the linear independence of the eigenfunctions in the expression $\psi\cdot f(\epsilon) = \lambda\psi$, we find that
$$f_{12}(p(r-p+1)) = 0,\ \ f_{21}((r-p)(p+1)) = 0,\ \ f_{11}(p(r-p+1)) = f_{22}((r-p)(p+1)).$$
From this we see that 
$$D(\wt w,\nu,w)= \nu^{-1}\left\lbrace \left(\begin{array}{cc}
f_{11}(\epsilon) & f_{12}(\epsilon)\\
f_{21}(\epsilon) & f_{22}(\epsilon)
\end{array}\right) : \substack{f_{12}(p(r-p+1)) = 0,\ f_{21}((r-p)(p+1)) = 0,\\ f_{11}(p(r-p+1)) = f_{22}((r-p)(p+1))}\right\rbrace \nu.$$
Written another way,
$$D(\wt w,\nu,w) = \bbc I \oplus \nu^{-1}\mxx{q_1(\epsilon)}{q_1(\epsilon)}{q_2(\epsilon)}{q_2(\epsilon)}\nu \oplus \nu^{-1}M_2(q_1(\epsilon)q_2(\epsilon)\bbc[\epsilon])\nu.$$
for $q_1(\epsilon) = \epsilon - p(r-p+1)$ and $q_2(\epsilon) = \epsilon - (r-p)(p+1)$.
Since the leading coefficient of $\nu$ is nonsingular, conjugation by $\nu$ is order-preserving.  Furthermore, the order of the $2\times 2$ matrix $(f_{ij}(\epsilon)$ is $2\max_{ij} \deg(f_{ij})$.  Therefore we see that
$$\dim(D(\wt w,\nu, w)_i/D(\wt w,\nu, w)_{i-1}) = \left\lbrace\begin{array}{cc}
4, & \text{if $i>0$ is even}\\
0, & \text{if $i>0$ is odd}
\end{array}\right.$$
One may check by hand that there are no $\wt w$-symmetric operators of order $1$, and that the only operators of order $0$ in $D(\wt w)$ are $\bbc I$.  Thus by Proposition \ref{subalgebra proposition}, we have equality $D(\wt w) = D(\wt w,\nu, w)$.  
\end{proof}
It is evident from the presentation of the algebra $D(\wt w)$ given that $D(\wt w)$ is in fact isomorphic to the algebra found in our Hermite example (see Proposition \ref{hermite example presentation}).  In particular the center is a nodal cubic degenerating to a cuspidal cubic as $p\rightarrow n/2$.

The proof of the previous proposition shows us that $D(\wt w)$ is given by
$$D(\wt w) = \bbc I + \nu^{-1}\mxx{p(\epsilon-2)\bbc}{p(\epsilon-2)\bbc}{p(\epsilon)\bbc}{p(\epsilon)\bbc}\nu + \nu^{-1}M_2(\bbc[\epsilon]p(\epsilon)p(\epsilon-2))\nu,$$
A quick check shows that this algebra is generated by the four elements of order two.  Moreover, since the center of $M_2(\bbc[\epsilon]p(\epsilon)p(\epsilon-2))$ is $\bbc[\epsilon]p(\epsilon)p(\epsilon-2)I$, we have that the center of $D(\wt w)$ is
$$Z(\wt w) = \bbc I + \nu^{-1}\bbc[\epsilon]p(\epsilon)p(\epsilon-2)I\nu.$$
This is a commutative algebra whose spectrum is isomorphic to the nodal cubic curve $y^2 - 2(|b|^2+1)xy + |b|^2(|b|^2+2)x^2 - x^3 = 0$, via the isomorphism induced by the ring isomorphism defined by
$$\bbc[x,y]\rightarrow \bbc I + \nu^{-1}\bbc[\epsilon]p(\epsilon)p(\epsilon-2)I\nu,\ \ \ x\mapsto \nu^{-1} p(\epsilon)p(\epsilon-2)I\nu,\ y\mapsto \nu^{-1}\epsilon p(\epsilon)p(\epsilon-2)I\nu.$$
The ring $D(\wt w)$ has the structure of a module over its center, and as a module it is finitely generated and torsion-free.  However, it is not a free module over its center since
$$\mxx{0}{0}{p(\epsilon)}{0}(p(\epsilon)p(\epsilon-2)^2I) + \mxx{0}{0}{p(\epsilon)p(\epsilon-2)}{0}(-p(\epsilon)p(\epsilon-2)I) = 0I.$$
This completes our example.

\section{General Structure Results}
For the final section of the paper, we will introduce some results regarding the algebraic structure of $D(w)$ under fairly general assumptions.  To prove the desired results, we will use the fact that $D(w)$ is closed under an anti-involution $\dag$ and that $D(w)$ may be embedded into a matrix algebra and is therefore a PI-ring.  This latter fact will allow us to apply the results of \cite{sarraille1982module} and \cite{rowen1973some}.

Before proving our main result, we establish a very helpful lemma.
\begin{lem}
The algebra $D(w)$ is a semiprime $PI$-algebra.  The nonzero, $w$-symmetric elements of $D(w)$ are not nilpotent.
\end{lem}
\begin{proof}
Let $\{p(x,n)\}$ be a sequence of orthogonal polynomials for $w$.  We claim that if $\delta\in D(w)$ is a nonzero element, then $\delta\delta^\dag$ is also nonzero.  To see this, suppose that $\delta\neq 0$.  Then there exists an integer $n\geq 0$ such that $p(x,n)\cdot\delta\neq 0$.  It follows that
$$0\neq \langle p_n\cdot \delta,p_n\cdot \delta\rangle_w = \langle p_n\cdot\delta\delta^\dag,p_n\rangle.$$
Hence we have that $p_n\cdot\delta\delta^\dag$ is nonzero, and therefore $\delta\delta^\dag\neq 0$.

Next suppose that $\eta\in D(w)$ is a nonzero, nilpotent, $w$-symmetric element.  Then there exists a least integer $m>0$ satisfying $\eta^m = 0$.  Clearly $m>1$.  If $m$ is even, then we may write
$$0 = \eta^m = \eta^{m/2}\eta^{m/2} = \eta^{m/2}(\eta^{m/2})^\dag.$$
By the result of the previous paragraph, this means $\eta^{m/2}=0$, and since $0 < m/2 < m$ this contradicts the minimality of $m$.  Therefore $m$ must be odd.  However then $\eta^{m+1}=0$ and therefore by the same argument $\eta^{(m+1)/2} = 0$.  Since $m>1$ we have that $0 < (m+1)/2 < m$ so this again contradicts the minimality of $m$.  Since $m$ is neither even nor odd, this is a contradiction.  We conclude that nonzero, $w$-symmetric elements of $D(w)$ cannot be nilpotent.

Lastly, suppose that $\sheaf I$ is nontrivial nilpotent two-sided ideal of $D(w)$.  Since $\sheaf I$ is nontrivial, we may choose $\delta\in\sheaf I$ with $\delta\neq 0$.  Therefore $\delta\delta^\dag\in I$ is also nonzero.  However, since $I$ is nilpotent $\delta\delta^\dag$ should also be nilpotent, contradicting the result of the previous paragraph.  Hence $D(w)$ has no nontrivial, nilpotent, two-sided ideals.  This shows that $D(w)$ is semiprime.  Lastly, the fact that the eigenvalue homomorphism embeds $D(w)$ in $M_N(\bbc[n])$, combined with the fact that a matrix algebra is a PI-algebra shows that $D(w)$ is a PI-algebra.
\end{proof}

We are now ready to state and prove our general result on the structure of $D(w)$.
\begin{thm}
Suppose that $D(w)$ contains a differential operator of positive order with nonsingular leading coefficient.  Then the algebra $D(w)$ is finitely generated as a module over its center $Z(w)$ and $Z(w)$ is a reduced algebra of Krull dimension $1$.
\end{thm}
\begin{proof}
The Krull dimension of $D(w)$ is bounded by the GK-dimension of $D(w)$.  Consider the image $E(w)$ of $D(w)$ under the eigenvalue homomorphism $\Sigma :D(w)\cong E(w)\subseteq M_N(\bbc[n])$.  The Krull dimension of $E(w)$ is bounded by the GK-dimension of $E(w)$ as a graded vector space, graded by degree in the variable $n$.  Since $E(w)$ is a graded subalgebra of $M_N(\bbc[n])$, the GK-dimension of $E(w)$ is further bounded by the GK-dimension of $M_N(\bbc[n])$.  However, this latter algebra is Morita equivalent to $\bbc[n]$ and therefore has the same GK-dimension as $\bbc[n]$.  Since the GK-dimension and Krull dimension of a commutative ring agree, this means that the GK-dimension of $M_N(\bbc[n])$ is one.  Therefore the Krull dimension of $D(w)$ is at most one.  Since $D(w)$ contains a differential operator of order at least one, it must contain a $w$-symmetric differential operator $\delta$ of positive order $d$ with nonsingular leading coefficient.  Then $\delta^n$ has order $nd$ for each integer $n>0$, and it follows that $\delta$ is transcendental over $\bbc$.  Hence the Krull dimension of $D(w)$ is at least $1$.  We conclude that the Krull dimension of $D(w)$ is exactly $1$.  Combining this with the fact that $D(w)$ is a semi-prime PI-algebra, the main result of \cite{sarraille1982module} tells us that $D(w)$ is finitely generated as a module over its center.  Since $D(w)$ has Krull dimension $1$, it follows that $Z(w)$ also has Krull dimension $1$.

To show that $Z(w)$ is reduced, suppose that $\delta\in Z(w)$.  Then $\delta^\dag\in Z(w)$.  To see this, suppose that $\eta\in Z(w)$.  Then $\eta^\dag\delta = \delta\eta^\dag$ conjugating everything we find $\delta^\dag\eta = \eta\delta^\dag$.  Since $\eta\in D(w)$ was arbitrary, this shows that $\eta\in Z(w)$.  It follows that $\delta\delta^\dag\in Z(w)$, and by our previous lemma $\delta\delta^\dag$ is not nilpotent.  Since $(\delta\delta^\dag)^m = \delta^m(\delta^\dag)^m$, it follows that $\delta^m$ is also not nilpotent.  Since $\delta\in Z(w)$ was arbitrary, this shows that $Z(w)$ is reduced.
\end{proof}

In the specific case that $N=2$, we can say even more about the center $Z(w)$ of $D(w)$, namely that it is rational.
\begin{thm}
Suppose that $D(w)$ contains a differential operator of positive order with nonsingular leading coefficient.  If $D(w)$ is noncommutative and $N = 2$, then the center $Z(w)$ of $D(w)$ has a spectrum isomorphic to a rational curve.
\end{thm}
\begin{proof}
The previous theorem tells us that the spectrum of $Z(w)$ is a reduced (affine) curve.  To prove that $Z(w)$ is rational, L\"{u}roth's Theorem tells us that it suffices to find a ring monomorphism $Z(w)\hookrightarrow\bbc[n]$.  We will do so by using the eigenvalue homomorphism $\Sigma: D(w)\cong E(w)\subseteq M_N(\bbc[n])$.  Since $D(w)\neq Z(w)$ we may choose operators $\delta_1,\delta_2\in D(w)$ such that $\delta_1\delta_2\neq\delta_2\delta_1$.  Moreover we may choose $\delta_1$ and $\delta_2$ to be $w$-symmetric.  Then $\theta := i(\delta_1\delta_2-\delta_2\delta_1)$ is a nonzero $w$-symmetric differential operator and hence cannot be nilpotent.  Setting $d_i(n) := \Lambda(\delta_i)$ and $t(n) := \Lambda(\theta)$, we see that $t(n)= i(d_1(n)d_2(n) - d_2(n)d_1(n))$ is a trace-free, non-nilpotent matrix in $M_N(\bbc[n])$.  Since $t(n)$ is trace-free, $t(n)^2 = \det(t(n))$.  Thus since $t(n)$ is not nilpotent, $\det(t(n)))$ is not identically zero.  Thus for all but finitely many values of $n$ $d_1(n)d_2(n)-d_2(n)d_1(n)$ is a nonsingular matrix.  Therefore $d_1(n),d_2(n)$ generates the full matrix ring $M_2(\bbc)$ for all but finitely many values of $n$ by \cite{aslaksen2009generators}.  It follows that anything in $\Lambda(Z(w))$ must commute with all of $M_2(\bbc)$, and therefore that $\Lambda(Z(w))\subseteq \bbc[n]I$.  Hence $Z(w)$ has rational spectrum.
\end{proof}

\section{Conclusions}
The goal of this paper is to provide some insight into the rich algebraic structure of algebras of matrix differential operators associated to weight matrices.  We have demonstrated a method of constructing numerous examples of solutions to Bochner's problem for matrix differential operators from known solutions.  Using our method, we have also demonstrated how one may explicitly calculate the associated sequence of monic orthogonal polynomials, as well as the corresponding algebra of differential operators.

We've also included some general results regarding the algebraic structure of $D(w)$.  There are many interesting unanswered questions in this vein that an algebraically inclined individual could explore.  Examples include
\begin{itemize}
\item  Is $D(w)$ equal to the centrilizer in $\mweyl$ of one of its elements?
\item  Is $D(w)$ isomorphic to the endomorphism ring of a module over its center?
\item  Is $D(w)$ maximal among finitely generated algebras over $Z(w)$ with center $Z(w)$?
\item  When is the center of $D(w)$ rational?
\item  Is there a stacky picture we can associate to $D(w)$, akin to what we do for hereditary orders?
\end{itemize}
We hope to address some of these questions in future work.

\section*{Acknowledgement}
During the course of this paper, the author had numerous helpful discussions with Max Lieblich and S. Paul Smith.  The author would especially like to thank S. Paul Smith for bringing the papers \cite{tirao2011} and \cite{castro2006} to the authors attention, and for posing the original question of simplifying the results therein.  The author would like to thank Erik Koelink and Pablo Roman for pointing out several errors in an earlier version of this paper, and for suggesting the application of the theory developed therein to the $2\times 2$ Gegenbauer weight matrix studied by Ignacio Zurri\'{a}n in \cite{zurrian2015}\cite{zurrian2016algebra}.  The author would like to thank F. Alberto Gr\"{u}nbaum for his encouragement and remarks on an earlier draft.  The author would also like to thank the work of two anonymous referees, whose remarks helped greatly improved a previous version of this paper.

\bibliographystyle{plain}
\bibliography{omp}

\end{document}